\documentclass[11pt,a4paper]{amsart}

\usepackage{graphics, amsmath, amsfonts, amssymb, amsthm, amscd, mathrsfs, color}
\input xy

\usepackage[mathscr]{eucal}
 \usepackage{extarrows}

\theoremstyle{plain}
\newtheorem{Thm}{Theorem}[section]

\newtheorem{Lem}[Thm]{Lemma}
\newtheorem{Prop}[Thm]{Proposition}
\newtheorem{Prob}[Thm]{Problem}

\theoremstyle{definition}
\newtheorem{Def}[Thm]{Definition}

\theoremstyle{remark}
\newtheorem{Rem}[Thm]{Remark}
\newtheorem{Exa}[Thm]{Example}

\newcommand{\Aut}{ \operatorname{{\rm Aut}}}
\newcommand{\SAut}{ \operatorname{{\rm SAut}}}
\newcommand{\Lie}{ \operatorname{{\rm Lie}}}
\newcommand{\VF}{ \operatorname{{\rm VF}}}
\newcommand{\Gl}{ \operatorname{{\rm GL}}}
\newcommand{\Sl}{ \operatorname{{\rm SL}}}
\newcommand{\SU}{ \operatorname{{\rm SU}}}
\newcommand{\U}{ \operatorname{{\rm U}}}
\newcommand{\SO}{ \operatorname{{\rm SO}}}
\newcommand{\OO}{ \operatorname{{\rm O}}}
\newcommand{\T}{ \operatorname{{\rm Tame}}}
\newcommand{\Aff}{ \operatorname{{\rm Aff}}}
\newcommand{\Jonq}{ \operatorname{{\rm Jonq}}}
\newcommand{\Osh}{ \operatorname{{\rm Osh}}}
\newcommand{\Hol}{ \operatorname{{\rm Hol}}}
\newcommand{\Cont}{ \operatorname{{\rm Cont}}}

\newcommand{\udisk}{\mathbb{D}} 
\newcommand{\id}{\mathrm{id}}
\newcommand{\Eeul}{\EuScript{E}}

\newcommand{\mat}[3]{{\mathrm{Mat}\left( #1 \times #2 ; \, #3 \right)}} 

\newcommand{\CVF}{ \operatorname{{\rm CVF}}}
\newcommand{\Iso}{ \operatorname{{\rm Iso}}}
\newcommand{\reg}{ {\rm reg}}
\newcommand{\Id}{ {\rm Id}}
\newcommand{\codim}{ {\rm codim}}
\newcommand{\boundary}{\partial} 

\newcommand{\g }{{\mathfrak{g}}}

\newcommand{\B}{\ensuremath{\mathbb{B}}}
\newcommand{\Z}{\ensuremath{\mathbb{Z}}}
\newcommand{\R}{\ensuremath{\mathbb{R}}}
\newcommand{\N}{\ensuremath{\mathbb{N}}}
\newcommand{\C}{\ensuremath{\mathbb{C}}}
\newcommand{\G}{\ensuremath{\mathbb{G}}}
\newcommand{\cO}{{\ensuremath{\mathcal{O}}}}
\newcommand{\cat}{{\ensuremath{/\! /}}}

\numberwithin{equation}{section}

\begin{document}

\title[Manifolds with infinite dimensional group of holomorphic  automorphisms] 
{Manifolds with infinite dimensional group of holomorphic  automorphisms and 
the Linearization Problem}
\author{Frank Kutzschebauch}
\address{Departement Mathematik\\
Universit\"at Bern\\
Sidlerstrasse 5, CH--3012 Bern, Switzerland}
\email{frank.kutzschebauch@math.unibe.ch}
\thanks{Kutzschebauch partially supported by Schweizerischer Nationalfonds Grant 200021-116165 }
\date{\today}
\bibliographystyle{amsalpha}
\begin{abstract} 
We overview a number of precise notions for a holomorphic automorphism group to be big together with their implications,  in particular we give an exposition of the notions of flexibility and  of density property.

These studies have their origin in the famous result of Anders\'en and Lempert from 1992 proving that the overshears generate a dense subgroup
in the holomorphic automorphism group of $\C^n, n\ge 2$. There are many applications to natural geometric questions in complex geometry, several of which we mention here.

Also the Linearization Problem, well known since the 1950’s and considered by many authors, has had a strong influence on those studies.
It asks whether a compact subgroup in the holomorphic automorphism group of $\C^n$ is necessarily conjugate to a group of linear automorphisms. Despite many positive results,
the answer in this generality is negative as shown by Derksen and the author. We describe various developments around that problem.

\end{abstract}

\maketitle

\tableofcontents

\section{Introduction}
The holomorphic and at the same time algebraic automorphism group of the complex line $\C$ consists of invertible 
affine maps $z\mapsto az+b, a\in \C^\star, b\in \C$. It is a complex Lie group generated by the translations
 $z \mapsto z +b$ and the rotations $z \mapsto  a z$.

It is a classical fact that $\C^n$ for $n\ge 2$ has  infinite dimensional groups of algebraic and of
holomorphic automorphisms. Indeed, maps of the form

\begin{equation} \label{shear}
 (z_1, \ldots, z_n) \mapsto (z_1, \ldots , z_{n-1}, z_n + f (z_1, \ldots , z_{n-1})), 
\end{equation}

\noindent
where $f \in \cO (\C^{n-1})$ is an arbitrary polynomial or holomorphic   function of $n-1$ variables, 
are automorphisms. They can be viewed as time-1 maps of the vector field $\theta = f(z_1, \ldots, z_{n-1}) {\partial \over {\partial z_n}}$.
In complex analysis (when $f$ is holomorphic) such an automorphism is called a shear and such a vector field  is called a shear field. If $f$ is a 
polynomial  the complex analysts call the automorphism  a polynomial shear, whereas in affine algebraic geometry it is
 called an elementary automorphism.  The geometric idea behind this automorphism is to look at $\C^n$ as a trivial line bundle over $\C^{n-1}$ and 
 performing  a translation in each of the fibers of that line bundle, that depends polynomially or holomorphically on the base point.
In the same way one can use rotations depending on the base point

\begin{equation} \label{overshear} (z_1, \ldots, z_n) \mapsto (z_1, \ldots , z_{n-1},  f (z_1, \ldots , z_{n-1}) \cdot z_n), 
\end{equation}

\noindent
where $f \in \cO^*(\C^{n-1})$ is a nowhere vanishing holomorphic function. By simple connectedness of $\C^{n-1}$
the function $f$ is the exponential  $f = e^g$ of  some holomorphic $g \in \cO (\C^{n-1})$. Again such an automorphism  is the 
time-1 map of a complete(ly integrable) holomorphic vector field $\theta = g(z_1, \ldots, z_{n-1}) z_n {\partial \over {\partial z_n}}$. These  automorphisms are called overshears
and the corresponding vector fields  overshear fields. This notion  is not relevant for affine algebraic geometry, 
since nowhere vanishing polynomials on affine space are constant. These maps together with affine automorphism are the first obvious candidates for a generating set of the group
of holomorphic automorphisms $\Aut_{hol} (\C^n)$. We will come to this point in the next section.

The content of our article is an account of the holomorphic side of the problem of understanding the automorphism group of affine space and understanding how to detect affine space
among manifolds. 

The  attempts of getting a better understanding and possibly some structure theorems for the infinite dimensional  groups $\Aut_{hol} (\C^n)$ and $\Aut_{alg} (\C^n)$ started earlier
in the algebraic case  than  in the holomorphic category. We present some of the the  known results in the next section. 

One attempt of better understanding these groups was the Linearization Problem asking 
 about the ways a reductive (for example a finite) group can act on $\C^n$. Or equivalently, how to find all subgroups of the algebraic or the holomorphic automorphism 
 group of $\C^n$ isomorphic to a given reductive group $G$. The conjectured answer was that all such groups are conjugated into the group of linear transformations $\Gl_n (\C)$.

Another intriguing  question is whether this rich group of holomorphic or algebraic automorphisms can be used to  characterize affine space $\C^n$ holomorphically among Stein 
manifolds or algebraically among smooth affine algebraic varieties, i.e.,  subsets of $\C^N$ given by finitely many polynomial equations. Remember  that a Stein manifold is by definition a complex manifold that is holomorphically convex and holomorphically separable. Let be reminded that by the Bishop--Remmert embedding theorem any Stein manifold is a closed complex submanifold of some  $\C^N$ given by finitely many holomorphic equations. Thus Stein manifolds are the natural holomorphic analogue of affine algebraic manifolds. Surprisingly both problems Linearization and characterization of $\C^n$ are intimately related which is the reason why we present them both here. 

It was understood since a long time  that in order to solve the Linearization Problem one needs to be able to characterize $\C^n$. Most easily this is seen at the example of the famous Zariski Cancellation Problem (see Problem \ref{Zariski}) (the author learned this from Hanspeter Kraft, who says that several people knew it).  Suppose 
$X \times \C$ is  algebraically isomorphic (resp. biholomorphic) to $\C^{n+1}$ where $X$ is not algebraically isomorphic (resp. biholomorphic) to $\C^n$. Then the action of the group of 2 elements generated by
$\sigma: X\times \C \to X\times \C$ $(x,t) \mapsto (x, -t)$ acts on $\C^{n+1}\cong X\times \C$ with a fixed point set $X$ not algebraically isomorphic (resp. biholomorphic) to $\C^n$. Thus this action cannot be conjugate to a linear action
since the fixed point sets of linear actions are affine spaces. Thus to answer the linearization question to the positive it is necessary to solve Zariski's Cancellation Problem to the positive  and this in turn would  most likely require a certain characterization of $\C^n$. More  information in this direction can be found in section \ref{affine space}
 and subsection \ref{relations}.  When looking for a such a characterization of $\C^n$
a natural attempt is to use the automorphism group. It is a general phenomenon in mathematics and nature that highly symmetric objects are rare and such objects often can be characterized 
by their symmetries. As an instance in our subject one can name the fact that among bounded domains the unit ball $\B := \{ z=(z_1. z_2, \ldots , z_n) \in \C^n : \vert z_1 \vert^2 +
\vert z_2 \vert^2 + \ldots + \vert z_n \vert^2 < 1 \}$ can be characterized  among bounded strictly pseudoconvex domains by the dimension of its holomorphic automorphism group. Now the algebraic and holomorphic automorphism groups of  are large, at least infinite dimensional as we saw above. What would be a natural "largeness" condition which is exclusively satisfied by the holomorphic automorphism group of $\C^n $. The search for structure on $\Aut_{hol} (\C^n)$  lead to a nice guess, the density property (see the Varolin Toth Conjecture \ref{VTconjecture}). 
 More precisely, Anders\'en and Lempert   proved a remarkable result  which was further developed by Forstneri\v c and Rosay to the so called Anders\'en-Lempert Theorem \ref{AL}. This result shows in  a precise way how big 
$\Aut_{hol} (\C^n)$ is.  The crucial ingredient of the Anders\'en-Lempert Theorem (sometimes called the Anders\'en-Lempert Lemma) was then  naturally  generalized by Varolin to the notion of density property.  Now the search for Stein manifolds with this property started and it is still not clear at the moment whether this property (together with an obvious condition having the right structure as a differentiable manifold) characterizes $\C^n$.
The Anders\'en-Lempert Theorem  in turn was used to construct non-straightenable holomorphic embeddings which then in turn
led to the counterexamples to Holomorphic Linearization described in the last section. Thus the attempt of finding structure led to the study of largeness properties of $\Aut_{hol} (\C^n)$ and a nice application of the theory to embedding questions finally led to
the negative solution of the Linearization Problem. The details of this construction based on the ingenious trick of Asanuma are described in subsection \ref{counterex}.

Thus one  concrete question pointing in the direction of characterizing affine space,  was solved to the negative by another (not yet finished) attempt of charactering affine space. 
Welcome to the story.
 
The article is organized as follows. In Section 2 we give an account of the knowledge about the algebraic and holomorphic automorphism groups of affine space $\C^n$. The notions of density property and flexibility are explained in Section 3 together with their implications, most prominent for the density property this is the Anders\'en-Lempert Theorem \ref{AL-Theorem} and for flexibility this is the Oka principle \ref{Oka}. In Section 4 we give a list of beautiful applications of the theory behind these notions to
 natural geometric problems. In Section 5 we turn to the Holomorphic Linearization Problem. After explaining basic notions like the categorical quotient in Section 5.1. we go through the history in Section 5.2.. In Section 5.3. we explain one of the most spectacular applications of number (2) in the list of applications from Section 4: The counterexamples to Holomorphic Linearization found by Derksen and the author. In Section 5.4. we relate these counterexamples to famous Zariski Cancellation Problem and other long standing problems. In Section 5.5. we explain recent positive results on Holomorphic Linearization which very heavily point to the conjecture that the method of constructing counterexamples explained in Section 5.3. is the only possible method.

 We wish to thank the referee for a careful reading of the manuscript and giving valuable advice to improve the presentation.

\section{Complex Affine Space}\label{affine space}

From now on $n\ge 2$. As we have seen in the introduction there are plenty of  polynomial automorphisms of $\C^n$ given by formula (\ref{shear}).
Together with  the affine automorphisms $\Aff (\C^n)$ consisting of maps  $ Z \mapsto A Z +b$, where $A \in \Gl_n (\C)$ and $b \in \C^n$, they generate the group $\T (\C^n)$ of 
tame automorphisms. Remark that using conjugation with linear automorphisms permuting the coordinates we see that  all automorphisms of the form
 \eqref{shear} with any other variable $z_i$ playing the role of $z_n$   are contained in the group $\T (\C^n)$ of tame automorphisms. The natural
 question arising is whether all algebraic automorphisms are tame automorphisms, i.e., is $\Aut_{alg} (\C^n) = \T (\C^n)$?  Suggesting a  positive answer to that 
 question this was   known under the name "Tame Generator Conjecture". The conjecture was supported by  the classical positive answer for $\C^2$
 \begin{Thm} [Jung 1942 \cite{J}] 
 $\T (\C^2) = \Aut_{alg} (\C^2).$
 \end{Thm}
 
 However already in 1972 Nagata proposed a candidate for a counterexample in $\C^3$, now well known under the name  Nagata automorphism.
\begin{equation} \label{Nagata}
(z_1, z_2, z_3) \mapsto (z_1 - 2(z_1 z_3 + z_2^2) z_2 - (z_1 z_3 + z_2^2)^2 z_3, z_2 + (z_1 z_3 + z_2^2)z_3, z_3)	
\end{equation}
In order to see that this map is invertible it helps to observe that it preserves the polynomial $u =z_1 z_3 + z_2^2$. 

 It took until 2003 when Umirbaev and Shestakov \cite{US} were able to confirm that the Nagata automorphism is not
  tame. The Tame Generator Conjecture is still open for $n\ge 4$. In particular the Nagata automorphism is stably tame
   \cite{Smith}  more precisely the extension of the Nagata automorphism to $\C^4$ by mapping the fourth coordinate to itself is a tame automorphism of $\C^4$.. 
   
Let us also name a classical structure result for the algebraic automorphism group $\Aut_{alg} (\C^2)$ of $\C^2$.
To formulate it we introduce the group of triangular or elementary  automorphisms, also called the de Jonqui\`eres group:

\begin{equation} \label{Jonq}  \Jonq (\C^n) = \{ (\varphi = (\varphi_1, \ldots ,\varphi_n) \in \Aut_{alg} (\C^n) \ : \  \varphi_i = \varphi_i (z_1, \ldots, z_i) \}
\end{equation}

\begin{Thm} [van der Kulk 1953 \cite{vdK}] \label{vdKu}
The group $\Aut_{alg} (\C^2)$ is a free amalgamated product of the group of affine automorphisms $\Aff (\C^2)$ and the group of elementary 
automorphisms $\Jonq (\C^2)$ over their intersection.	
\end{Thm}

Inspired by the algebraic results and the Tame Generator  Conjecture   W. Rudin together with P. Ahern proved  the analogous structure result in the holomorphic category.
 Here the group generated by holomorphic shears (see equation \eqref{shear}), overshears (see equation \eqref{overshear})
and affine automorphisms is called the overshear group $\Osh (\C^n)$. The definition of the holomorphic 
de Jonqui\`eres group $\Jonq_{hol} (\C^n)$ is clear. For convenience of the reader we write their members for $\C^2$
$$ (z, w) \mapsto (az+b, e^{f(z)} w+ g(z))$$
where $a\in \C^*$, $b\in \C$ are numbers and  $f,g \in \cO (\C)$ holomorphic functions.

\begin{Thm} [Ahern, Rudin \cite{AR}] \label{Rudin}
The group $\Osh (\C^2)$ is a free amalgamated product of  $\Aff (\C^2)$ and    $\Jonq_{hol} (\C^2)$ over their intersection.	
\end{Thm}

Similar structure results for higher dimensions $n\ge 3$ do not hold. We invite the reader to write the identity 
automorphism as a non-trivial product of triangular and affine automorphisms. Only for two dimensional affine manifolds similar structure 
results are known. For example automorphism groups   of Danielewski surfaces admit structures of free amalgamated products. Danielewski surfaces  are the surfaces given by $D:= \{(u, v, z) \in \C^3 \ : \ u v = f(z)\}$, where $f$ is either a 
polynomial or a holomorphic function with simple zero's only. This  condition of simple zeros ensures the smoothness of the 
corresponding affine algebraic variety (resp. Stein manifold).
In the algebraic category these structure  results are due to Makar-Limanov \cite{MLDani} and in the holomorphic to 
Andrist, Lind and the author \cite{AKL}, see these reference for relevant definitions of the involved groups in the amalgamated product structure.

Rudin very much promoted the analogue of the Tame Generator Conjecture in the holomorphic category. Is $\Osh (\C^n)
= \Aut_{hol} (\C^n)?$

In their seminal paper Anders\'en and Lempert \cite{AL} gave two answers to this question, the first of them solving the
 problem in the negative has not been of much future importance. The group $\Aut_{hol} (\C^n)$ is a topological group when equipped 
 with the topology of uniform convergence on compacts, the usual c.-o.-topology. The following metric $d$ on $\Aut_{hol} (\C^n)$
 induces this topology and makes this topological space  a complete metric space, i.e., a Frech\' et space (\cite{AL}  Proposition 6.4.).

 For $\Psi, \Phi \in \Aut_{hol} (\C^n)$, and $r=1, 2, \ldots $ put
 
 \begin{equation} \begin{split}
 d_r (\Psi, \Phi ) = \max \{ \max_{ \vert z \vert \le r}  \vert \Phi (z) - \Psi (z) \vert, \max_{ \vert z \vert \le r}  \vert \Phi^{-1} (z) - \Psi^{-1} (z) \vert \}, \\
  d(\Psi, \Phi ) = \sum_{r=1}^\infty \min (1, d_r (\Psi, \Phi ) ) 2^{-r} \end{split}
  \end{equation}

 \begin{Thm} For all $n \ge 2$ the overshear group $\Osh (\C^n)$ is meagre in $\Aut_{hol} (\C^n)$.
 	\end{Thm}

Interestingly for $n\ge 3$ no concrete holomorphic automorphism is known which is not contained in $\Osh (\C^n)$.
It is for example not known whether the Nagata automorphism from formula $(\ref{Nagata})$ is contained in $\Osh (\C^3)$. 
In dimension $2$ it follows easily from the structure Theorem \ref{Rudin} that automorphisms  of the form $$(z_1, z_2) \mapsto 
(e^{f(z_1z_2)} z_1, e^{-f(z_1z_2)} z_2)$$ are not contained in $\Osh (\C^2)$ \cite{Andersen}, \cite{KK}.

The second answer Anders\'en and Lempert gave to Rudin's question  is of great value, it has been the starting point of an enormous development leading to many beautiful results 
of geometric nature in complex analysis. This still extremely active area will be described in the next section.

\begin{Thm}\label{AL} For all $n \ge 2$ the overshear group $\Osh (\C^n)$ is dense  in $\Aut_{hol} (\C^n)$ with respect 
to the
c.-o.-topology.
	\end{Thm}

This exemplifies the fact  that finite compositions are not the appropriate notion for complex analysis when it comes to 
infinite dimensional automorphism groups. Rather limits of finite compositions are the appropriate notion.

Finally let us come to the Holomorphic Linearization Problem that was long standing and could be solved only 
with the methods  of the area  which is described in the next section and which was  started by the seminal paper of Anders\'en and Lempert \cite{AL}.

In what follows $G$ will be a complex  reductive group. For us a complex  reductive group $G$ is a complex Lie group which is the universal 
complexification $G = K^\C$ of its maximal compact subgroup $K$. Examples of such groups are $\Gl_n (\C ) = \U_n^\C$, 
$\Sl_n 
(\C ) = \SU_n ^\C$ or even unconnected groups like $\OO_n (\C)$ which has $2$ connected components. Even finite groups are 
most interesting (which by definition are the universal complexification of themselves). Suppose a reductive group $G$ is 
acting (effectively)  holomorphically on $\C^n$, $n\ge 2$, in other words,  let a holomorphic  
 (injective) group homomorphism  $G\hookrightarrow \Aut_{hol} (\C^n)$ be given.

\begin{Prob}[\bf Holomorphic Linearization Problem] \cite{Huck}

Does there exist a holomorphic change of variables $\alpha \in \Aut_{hol} (\C^n)$ such that $\alpha \circ G \circ \alpha^{-1} $ is linear?
\end{Prob}
If the answer is positive, we say the action is linearizable. The problem can be equivalently formulated for compact (real-analytic Lie) groups instead
due to the following result of the author. Using the discussion about completeness of vector fields at the beginning of section \ref{defandmain} one can generalize this result from $\C^n$ to any Stein manifold with the density property.

\begin{Thm}  \cite{K97} \label{complex} If a real Lie group $H$ acts on $\C^n$ by holomorphic transformations, then this action extends (uniquely) to a holomorphic action
of the complexification $H^\C$.
\end{Thm}

Thus an  action of a compact group $K$ extends to an action of the reductive group  $G= K^\C$ and is linearizable, iff the action of $G$ is. This follows directly from the identity principle for holomorphic mappings since $K$ is totally real of maximal dimension in $G$.

As a sort of curiosity, it also follows  from the above theorem that the real Lie groups
 $\widetilde {\Sl_k (\R) }$ for $k\ge 2$  
 (the universal covering of $\Sl _k ( \R ) $) cannot act effectively on $\C^n$ since they are not injective into their  universal complexifications $\Sl_k (\C)$.

The linearization  question in the algebraic setting is classical and has drawn much attention. We don't want to go into these details, we just want to explain the counterexamples to the Algebraic Linearization Problem first constructed by Gerald Schwarz \cite{Schwarz89}.
They are known as $G$-vector bundles over representations. A representation is  a $\C^k$ with a linear action (representation) of $G$, i.e., 
a group homomorphism $ \alpha : G \to \Gl_k (\C)$. By the Quillen-Suslin solution to the Serre question, every algebraic vector bundle  (say of rank
 $l$) over $\C^k$ is trivial, i.e., isomorphic to $\C^k \times \C^l \to \C^k$. Now we require the linear $G$-action on the base $\C^k$ to lift to the total space
 as  vector bundle automorphisms. This means there is a map $\beta :  G \times \C^k \to \Gl_l (\C)$ such the action of $G$ on the total space 
 $\C^n = \C^k \times \C^l$ of the bundle, which is a map $G \times (\C^k \times \C^l) \to  (\C^k \times \C^l)$, will be  given by the formula
 \begin{equation}
  (g, (z, w)) \mapsto (\alpha(g) \cdot z, \beta (g, z) \cdot w)
  \end{equation}
  and the $\cdot$ is multiplication of a vector by a matrix. Clearly the map $\beta$ has to satisfy some obvious rules coming from the requirements of  a 
  group action. The only non-linearity in this sort of action is the dependence of the map $\beta$ on the ``base point" $z$. A concrete example used by
   G. Schwarz is the following:
   
   \begin{Exa} \label{counter} We will construct a non-linearizable action of  the group $G = \C^* \times \OO_2 (\C)$ on $\C^4$. For that we consider $\C^4$ as a trivial vector bundle 
   of rank $2$ over $\C^2$. First we describe the  action of $\OO_2 (\C) \cong  \SO_2 (\C) \rtimes  \Z/2\Z \cong \C^\star \rtimes  \Z/2\Z$ by vector bundle automorphisms of  the trivial vector bundle. For that let $\lambda$ denote the variable in $\SO_2 (\C) \cong \C^\star$ and $\sigma$ the generator of $\Z/2\Z$. Then the action of $\lambda$ on the trivial vector bundle  $\C^ 2 \times \C^2$ with coordinates $((x,y), (u,v))$  is given by

   $$ (\lambda, (u,v), (x,y)) \mapsto ((\lambda^2 u, \lambda^{-2} v), (\lambda^3 x, \lambda^{-3} y))$$
   and the action of $\sigma$ is given by 
   $$    ((u,v), (x,y)) \mapsto((u, v), ((1+uv+(uv)^2) y - v^3 x,    u^3 y +(1-uv) x      ))$$
   
   The action of $\SO_2 (\C)$ is linear however the element $\sigma$ makes this bundle non-trivial as an $\OO_2 (\C)$-bundle. The action of the additional factor $\C^\star$ in $G$ is
   just the multiplication in the fibre. This has the consequence, that a $\C^\star$-equivariant isomorphism is nothing than a bundle isomorphism. Since there is no linearization of the $\SO_2 (\C)$-action
   by bundle isomorphisms, we have a non-linearizable action of $G$. The reader is invited to linearize the bundle, and thus the action, holomorphically by explicit formulas involving the exponential function.
   \end{Exa}

 Already   Schwarz   mentioned that this particular action is {\bf holomorphically} linearizable. This is an instance of a more general result described in section \ref{Linearization}, based on the Equivariant Oka Principle  (Theorem \ref{thm:HK-classification}.). Let us just finish the discussion of the Algebraic Linearization Problem by saying that there are non-linearizable actions of all semi-simple groups, that no counterexamples for
abelian groups are known yet, and that actions of connected reductive groups on $\C^2$    and on $\C^3$ \cite{Kraft} are known to be linearizable. For $\C^2$ this follows classically from Theorem \ref{vdKu}. The most striking case was that of 
   $\C^*$-actions on $\C^3$, a result with very indirect proof \cite{KMKZKR}. First came  the classification of all contractible
   affine algebraic $3$-folds with $\C^*$-action. Then it  remained to prove that some of those varieties, which if they were isomorphic to $\C^3$ would carry a non-linearizable action, are in fact not isomorphic to $\C^3$. The most famous example is the so called Koras-Russell cubic threefold defined by the equation
   \begin{equation} \label{KR} M_{KR} := \{ (x,y,s,t) \in \C^4 \ : \ x + x^2y + s^2+ t^3 = 0 \} \end{equation}

It was Makar-Limanov who could distinguish this variety algebraically from $\C^3$ by inventing his famous Makar-Limanov invariant \cite{Makar}.

We will continue this discussion of the Linearization Problem in the last section, but then concentrating more on  the holomorphic setting.

\section{Notions for  the largeness of an automorphism group}

A natural way of saying  that $\Aut_{hol} (X) $ of a complex manifold is large is that it is infinite-dimensional. This is by definition the case 
 if the Lie algebra (and equivalently the vector space)  generated by $\R$-complete holomorphic vector fields on $X$ is infinite-dimensional. 
  This condition is equivalent to the impossibility
of introducing on $\Aut_{hol} (X) $ a topology with respect to which it becomes
a Lie transformation group (possibly, with uncountably many connected
components) in the sense of Palais (\cite{Palais},
  p.p. 99, 101, 103). In particular, in this case $\Aut_{hol} (X) $ is not a Lie group in the
compact-open topology and not even locally compact (see \cite{MZ}, p. 208).  
  
  The notions we will introduce in this section are much stronger (or more precise) and imply the infinite-dimensionality of the holomorphic automorphism group.
For example consider the manifold $X = \udisk \times \C$ which is the product of the unit disc $\udisk = \{ z \in \C \ : \ \vert z \vert < 1\} $ with the complex line. It has an infinite-dimensional group of holomorphic
automorphisms. Indeed  it contains the shear-like maps $(z, w) \mapsto (z, f (z) + w)$ for all holomorphic functions $f \in \cO (\udisk)$. On the other hand this manifold does not have the notions we will introduce in this chapter, it does not have the density property and  is not   holomorphically flexible. The same holds true for the direct product of any Kobayashi hyperbolic manifold with $\C$.

Another class of examples of manifolds with infinite-dimensional holomorphic automorphism group are homogeneous Stein manifolds of dimension $\ge 2$ as proved by Huckleberry and Isaev \cite{HI}.
In contrary to the example  $\udisk \times \C$, it seems possible that the Stein homogeneous spaces of dimension $\ge 2$, except possibly $(\C^\star) ^n$, all have the density property. If one assumes that the homogeneous space is affine algebraic (instead of assuming Stein) this is a recent result of Kaliman
and the author, see Example (1) in the list of examples with density property from subsection \ref{examples}.

\subsection{Density property}

Considering a question promoted by Walter Rudin, Anders\'en and Lempert in 1992 \cite{AL} proved a remarkable fact about the affine $n$-space $n\ge 2$, namely that the group generated by shears (maps of the form $$(z_1, \ldots, z_n) \mapsto (z_1, \ldots , z_{n-1}, z_n + f (z_1, \ldots , z_{n-1}))$$ where $f \in \cO (\C^{n-1})$ is a holomorphic function and any linear conjugate of such a map) and overshears (maps of the form $$(z_1, \ldots, z_n) \mapsto (z_1, \ldots , z_{n-1}, z_n  g (z_1, \ldots , z_{n-1}))$$ where $g \in \cO^\ast (\C^{n-1})$ is a nowhere vanishing holomorphic function (and any linear conjugate of such a map) are dense in holomorphic automorphism group of $\C^n$, endowed with compact-open topology.

The main importance of their work was not the mentioned   result but the proof itself which implies, as observed by Forstneri\v c and Rosay in \cite{FR} for $ X= \C^n$, the  remarkable  Anders\'en--Lempert theorem, see below. The natural generalization from $\C^n$ to arbitrary manifolds $X$ was made by Varolin \cite{V1} who introduced the  important
notion of the density property.

\subsubsection{Definition and main features} \label{defandmain}

Recall that a  holomorphic vector field $\Theta$ on a complex manifold $X$  is called complete or completely integrable if the ODE

$$\frac{d}{dt} \varphi (x,t) = \Theta (\varphi (x,t))$$

$$\varphi (x,0) = x$$

\noindent
has a solution $\varphi (x,t)$ defined for all complex times $t \in \C$ and all starting points $x \in X$. It gives  a complex one-parameter subgroup in the holomorphic
automorphism group $\Aut_{hol} (X)$ (by definition this is a holomorphic group homomorphism $\C \to \Aut_{hol} (X)$) whose elements $\varphi_t = \varphi (\cdot, t) : X \to X$ we call flow maps or time-$t$-maps of the field $\Theta$.

We will denote the set of completely integrable holomorphic vector fields by $\CVF_{hol} (X)$ and the smallest Lie algebra 
containing a subset $A$ of the Lie algebra of all  vector fields on $X$ will be denoted by $\Lie (A)$. We call it the Lie algebra generated by $A$.

A remark about completeness of vector fields is in order. The just defined notion is usually called $\C$-complete. When the above ODE has a solution for all real times $ t\in \R$ a field is called $\R$-complete and when it 
has a solution for positive real times, the field is called  $\R^{+}$-complete. On Kobayashi hyperbolic
manifolds, for example the unit ball in $\C^n$, no $\R$-complete field can be $\C$-complete, since  entire holomorphic maps to such manifolds are constant. Also on the unit ball, there exist $\R^{+}$-complete holomorphic fields which are not $\R$-complete,  e.g., the Euler vector field contracting the ball to zero.

In contrast to this situation one could define the density property  (next definition)
using $\R$-complete holomorphic vector fields and derive the fact that on Stein manifolds with the density property any $\R$-complete holomorphic vector field is automatically $\C$-complete. The main ingredient in the proof of this fact  is that
on Stein manifolds with density property bounded plurisubharmonic functions are constant (see property (1)  of Stein manifolds with DP at the end of this subsection). Thus by the theorem in \cite{AFR} even all  $\R^{+}$-complete holomorphic fields on such manifolds are automatically $\C$-complete.  This is the reason why we use the simplified notion "complete" or "completely integrable" holomorphic vector field. Following the reasoning along these lines  one can prove  that Theorem
\ref{complex} is holds true when $\C^n$ is replaced by any Stein manifold with density property.

\begin{Def}\label{1.10}
A complex manifold $X$ has the density property (for short DP) if in the
compact-open topology the Lie algebra  $\Lie (\CVF_{hol} (X))$ generated by
completely integrable holomorphic vector fields on $X$ is dense in
the Lie algebra $\VF_{hol} (X)$ of all holomorphic vector fields on
$X$.
 \end{Def}

The density property is a precise way of saying that the automorphism group of a manifold is big, in particular for a Stein manifold this is underlined by the main result of the theory (see \cite{FR} for $\C^n$, \cite{V1},  a detailed proof can be found in the Appendix of \cite{R} or in  \cite{For}).

\begin{Thm}[\bf Anders\'en-Lempert Theorem] \label{AL-Theorem} Let $X$ be a Stein manifold with the density  property
and let  $\Omega$ be an open subset of $X$. Suppose that  $ \Phi : [0,1] \times \Omega \to X$ is a $C^1$-smooth
map  such that

{\rm (1)} $\Phi_t : \Omega \to X$ is holomorphic and injective   for every  $ t\in [0,1]$,

{\rm (2)} $\Phi_0 : \Omega \to X$ is the natural embedding of $\Omega$ into $X$, and

{\rm (3)} $\Phi_t  (\Omega) $ is  a Runge  
subset\footnote{Recall that an open subset $U$ of $X$ is Runge if any holomorphic function on $U$ can
be approximated by global holomorphic functions on $X$ in the compact-open topology. Actually, for $X$ Stein by Cartan's Theorem A this definition
implies more: for any coherent sheaf on $X$ its section over $U$ can be approximated by global sections.}  of $X$ for every $t\in [0,1]$.

Then for each $\epsilon >0 $ and every compact subset $K \subset \Omega$ there is a continuous family,
$\alpha: [0, 1] \to \Aut_{hol} (X)$  of  holomorphic  automorphisms of $X$
such that  $$\alpha_0 = id \, \, \, {\rm and} \, \, \,\vert \alpha_t - \Phi_t \vert_K <\epsilon {\rm \ \  for\ \  every\ \ } t \in [0,1]$$ 
Moreover, given a collection $A$ of completely integrable  vector fields such that $\Lie(A)$ is dense in $\VF_{hol} (X)$, the automorphisms $\alpha_t$ can be chosen to be  finite compositions of flow maps of vector fields from the collection $A$.

\end{Thm}

Philosophically one can think of the density property as a tool for realizing  local movements  by  global  maps   (automorphisms).    In some sense it is a substitute  for cutoff functions which in the differentiable category are used for globalizing local movements. In the holomorphic category we of course lose control on automorphisms
outside the compact set $K$. This makes constructions more complicated. Nevertheless constructing  sequences of automorphisms  by iterated use of the Anders\'en-Lempert theorem
has led to remarkable applications.

Let us further remark that the implications of the density property for manifolds which are not Stein have not been explored very much yet. If the manifold is compact all (holomorphic) vector fields are completely integrable, the
density property trivially holds and thus cannot give any further information on the manifold.

\begin{Rem} Anders\'en and Lempert proved that every algebraic vector field on $\C^n$ is a finite sum of  algebraic shear fields (fields of form $ p(z_1, \ldots z_{n-1}) \frac \partial {\partial z_n}$ for a polynomial $p \in \C[\C^{n-1}]$ and their linear conjugates, i.e.,  fields whose one-parameter subgroups consist of shears, see equation \eqref{shear})  and overshear fields 
(fields of form $ p(z_1, \ldots z_{n-1}) z_n  \frac \partial {\partial z_n}$ for a polynomial $p \in \C[\C^{n-1}]$ and their linear conjugates, i.e.,  fields whose one-parameter subgroups consist of overshears, see equation \eqref{overshear}). 
Here an algebraic vector field is an algebraic section of the tangent bundle, for example on  $\C^n$ it can be written as $\sum_{i=1}^n p_i (z_1, \ldots , z_n) \frac{\partial}{\partial z_i}$ with polynomials $p_i \in \C[\C^n]$.

Since algebraic vector fields are dense in c.-o.-topology  in the holomorphic vector fields, their result shows that 
$\C^n$ has the density property (they do not even need Lie brackets).
Together with the simple fact that any holomorphic automorphism of $\C^n$ can be joined to the identity by a smooth path, this shows how the Anders\'en-Lempert theorem implies that the group $\Osh(\C^n)$ is dense in the holomorphic automorphism group of $\C^n,$ (Theorem \ref{AL} above).
\end{Rem}

\begin{Def}
An affine algebraic manifold $X$ has the algebraic density
property (for short ADP) if the Lie algebra $\Lie (\CVF_{alg}) (X)$ generated by completely
integrable algebraic vector fields on it coincides with  the Lie
algebra $\VF_{alg} (X)$ of all algebraic vector fields on it.
\end{Def}

The algebraic density property for affine algebraic manifolds can be viewed as a tool to prove the density property, whereas the ways of proving it are purely algebraic work.
Since on an affine algebraic manifold the algebraic vector fields are dense in c.o.-topology in the holomorphic vector fields, the algebraic density property  implies the density property: ADP $\Longrightarrow $ DP.

 If an algebraic vector field is completely integrable, 
its flow gives a one parameter subgroup in the holomorphic automorphism group not necessarily in the algebraic automorphism group. For example,  a polynomial overshear field of the form $p (z_1, \dots , z_{n-1}) z_n  \frac{\partial}{\partial z_n}$
has the flow map $\gamma (t, z) = (z_1, \ldots , z_{n-1}, \exp (t p (z_1, \dots , z_{n-1}) ) z_n)$. This is the reason that the study of the  algebraic density property is in the intersection of affine algebraic geometry and complex analysis. It is an algebraic notion, proven using algebraic methods, but has implications for the holomorphic automorphism group.

\subsubsection{Criteria and Examples} \label{examples}

The first example of a manifold with the density property was $\C^n$ as explained above. Then Varolin himself found many examples \cite{Varolin2} and together with Toth he proved that semi-simple complex Lie groups and certain homogeneous spaces of them
have DP. Their proofs rely on results from representation theory and are therefore not applicable in a more general context. Later Kaliman and the  author found very effective criteria to prove DP, see e.g. \cite{KKinvent}, \cite{KK2017}. We will explain one of them now. The idea has two steps. 

\smallskip
STEP1: This crucial step  is a way of finding a submodule in $\VF_{hol} (X)$ contained in the closure of $\Lie(\CVF_{hol}(X))$. That is highly non-trivial, the main idea is the following:

\begin{Def}  A compatible pair is a pair $(\nu, \mu)$ of complete holomorphic vector fields such that the closure of the linear span of the product of the kernels $\ker \nu  \ker \mu$ contains a non-trivial ideal $I \subset \cO (X )$  and there is a function $h \in \cO (X)$ with $\nu (h) \in \ker \nu\setminus 0$ and $h\in \ker \mu$.  	The biggest ideal $I$ with this property will be called the ideal of the pair $(\nu, \mu)$.
\end{Def}

\begin{Lem} \label{compatible} 
Let  $(\nu, \mu)$ be a compatible pair, and let $I$ be its ideal and $h$ its function. Then the submodule of $\VF_{hol} (X)$ given by $I · \nu (h) · \mu$ is contained in the closure of $\Lie( \CVF_{hol} (X))$.
\end{Lem}

\begin{proof}
 Let $f \in  \ker \nu $ and $g \in  \ker\mu$, then $f \nu, fh\nu, g\mu, gh\mu \in  \CVF_{hol} (X)$. A standard calculation shows
$$ [f\nu ,gh\mu ] - [fh\nu ,g\mu ] = fg\nu(h)\mu  \in  \Lie(\CVF_{hol}(X)).$$
Thus an arbitrary element $\sum_{i=1}^N (f_i g_i)\nu(h)μ \mu \in  span (\ker\nu·\ker\mu)·\nu(h)·\mu$ with $f_i \in  \ker \nu$ 
and $g_i \in \ker \mu$ is contained in $\Lie(\CVF_{hol}(X))$ which concludes the proof. 
 \end{proof} 

STEP 2: Once some submodules are found, simple transitivity creates more of them as follows:

\begin{Def} Let $p \in  X$. A set $U \subset  T_p X$ is called generating set for $T_p X$ if the orbit of $U$ under the induced action of the stabilizer $\Aut_{hol}(X)_p$ contains a basis of $T_p X$.
\end{Def}

\begin{Prop} \label{moving} Let $X$ be a Stein manifold such that $Aut_{hol} (X)$ acts transitively on $X$. Assume that there are holomorphic vector fields $v_1,...,v_n \in  \VF_{hol} (X)$ which generate a submodule $\cO (X) v_i$ that is contained in the closure of $\Lie(\CVF_{hol} (X))$ and assume that there is a point $p \in  X$ such that the tangent vectors $v_i (p) \in  T_p X$ are a generating set for the tangent space $T_p X$. Then X has  DP.
\end{Prop}

\begin{proof} 
The basic observation is that the change of variables does not effect  the complete integrability of a vector field and commutes with the operations of forming the  Lie bracket and the sum of vector fields. Thus for any holomorphic automorphism $\alpha \in \Aut_{hol} (X)$ we have  $\alpha^* (\CVF_{hol} (X)) = 	\CVF_{hol} (X)$ and $\alpha^* (\Lie (\CVF_{hol} (X))) = 	\Lie (\CVF_{hol} (X))$ (and the same for the closure). Also if $L$ is a submodule of $\VF_{hol} (X)$ (as $\cO (X)$ module), then $\alpha^* (L)$ is again a submodule.

 We may assume that the vectors $v_i (p)$ contain a basis of $T_p X$. Indeed, the vectors $v_i (p)$ are a generating set of $T_p X$. Thus after adding some pullbacks of some vector fields $v_i$ by automorphisms in the stabilizer $\Aut_{hol} (X)_p$ we get the desired basis of $T_p X$. Let $A\subset X$ be the analytic subset of points where the vectors $v_i (a)$ do not span the whole tangent space $T_a X$.
Let $K \subset  X$ be a compact set. After replacing $K$ by its $\cO (X)$ -convex hull we may assume that $K$ is $\cO (X)$-convex. Let $Y$ be a neighborhood of $K$ which is Stein and Runge, and moreover such that the closure of $Y$ is compact. After adding finitely many pullbacks of the $v_i$ by automorphisms to the collection of  vector fields  $v_1, . . . , v_n$ we get that $Y \cap
A = \emptyset $. Indeed, since the closure of $Y$ is compact, $Y \cap A$ is a finite union of irreducible
analytic subsets. Let $A_0 \subset  A$ be an irreducible component of maximal dimension.
Pick any $a \in  A_0$ and $\varphi \in   \Aut_{hol} (X)$ such that $\varphi (a) \in  Y \setminus A$. Since the vectors
$v_i (\varphi (a))$ span the tangent space $T_{\varphi (a)} X$ the vectors $(\varphi^∗ ν_i )(a)$ span the tangent space 
$T_a X$. 
 Thus after adding some of the pull backs to $ v_1, . . . , v_n$ the component $A_0 \cap Y$ is replaced by finitely many components of lower dimension. Repeating the same procedure, inductively we get after finitely steps a list of complete vector fields $v_1,...,v_N$ such that $A\cap Y =\emptyset$.
 By the Nakayama lemma they generate the stalk of the tangent sheaf at each point in $Y$. By standard results  in the theory of coherent $\cO$-module sheaves  on Stein manifolds, the sum of the corresponding modules $\cO (X) v_i$ approximates every holomorphic vector field on the compact $K$.
Thus by the observation in the beginning of the proof every holomorphic vector field is in the closure of $\Lie(\CVF_{hol}(X))$.  
\end{proof}

Summarizing we obtain our criterion

\begin{Thm}
	 Let $X$ be a Stein manifold such that the holomorphic automorphisms $\Aut_{hol}(X)$ act transitively on $X$. If there are compatible pairs $(\nu_i, \mu_i)$ such that there is a point $p \in  X$ where the vectors $\mu_i (p)$ form a generating set of $T_p X$, then $X$ has the density property.
\end{Thm}

\begin{proof}
	Let $I_i$ be the ideals and $h_i$ the functions of the pairs $(\nu_i, \mu_i)$ and pick any non-trivial $f_i \in  I_i · \mu_i (h_i)$ for every $i$. Since the set of points $q \in  X$ where the vector fields $\mu_i (q)$ are a generating set is open and non-empty, there is some $p \in  X$ where the vector fields $f_i (p) \mu_i (p)$ are a generating set for $T_p X$. By Lemma \ref{compatible} the module generated by the vector fields $f_i \mu_i$ is contained in the closure of $\Lie(\CVF_{hol}(X))$ and thus by Proposition 
\ref{moving} the manifold $X$ has the density property.
\end{proof}

Before we  come to  the  complete  list of examples of Stein manifolds known to have the density property, we would like to 
show the power of our criterion in some examples:
\begin{Exa} 
On $\C^n$, $n\ge 2$ the pair of vector fields $({\partial \over \partial z_1} , {\partial \over \partial z_2})$ is compatible, with
function $z_1$ and $I = \cO (\C^n)$. Since we can permute coordinates, $\{ {\partial \over \partial z_2}\}$ is a generating set for each tangent space. Thus $\C^n$ has DP.
\end{Exa}
\begin{Exa}
On $X= \Sl_2 (\C )$ denote an element of $X$ by $$A =  \left(\begin{array}{ccc}
a_1 & a_2\\
b_1 & b_2 \end{array}\right) $$
The pair $(\delta_1, \delta_2)$ is compatible, where 

$$
\delta_1=a_1\frac{\partial}{\partial
b_1}+a_2\frac{\partial}{\partial b_2}
$$
$$
 \delta_2=b_1\frac{\partial}{\partial
 a_1}+b_2\frac{\partial}{\partial a_2} \, .
$$
with $I =\cO (X)$ and function $h = a_1$. These fields are corresponding to adding multiples of the first row to the second row and vice versa.
To see that $\delta_2$ is a generating set at the identity one can for example use the fact that the adjoint representation is irreducible. 
Thus $\Sl_2 (\C)$ has DP. The proof for $\Sl_n (\C)$ and $\Gl_n (\C)$,  $n\ge 3$ goes the same way.
\end{Exa}

\smallskip\noindent
\textbf{List of examples of Stein manifolds known to have the density property:}

\medskip\noindent
(1)  A homogeneous space $X=G/H$ where $G$ is a linear algebraic group and $H$ is a closed algebraic  subgroup  such that $X$ is affine
and whose connected component is different from $\C$ and from $(\C^*)^n, n\ge 1$
has DP. 

\medskip\noindent
 It is known that if  $H$ is reductive  the space $X=G/H$ is always affine, however there is no known group-theoretic characterization  which would 
say when $X$ is affine. This result has a long history, it  includes all examples known from the work of Anders\'en-Lempert and Varolin and Varolin-Toth and  Kaliman-Kutzschebauch, Donzelli-Dvorsky-Kaliman, the final result is proven by Kaliman and the author in  \cite{KK2017}. $\C$ and $\C^*$ do not have DP, however the following problem is well known and seems notoriously difficult

\medskip\noindent
\textbf{Open Problem:} Does $(\C^*)^n, n\ge 2$ have DP? 

\smallskip\noindent 
It is conjectured that the answer is no, more precisely one expects that
all holomorphic automorphisms of $(\C^*)^n, n\ge 2$ respect the form $\wedge_{i=1}^n { {\rm d} z_i \over  z_i}$ up to sign.

\medskip\noindent
(2) The manifolds  $X$ given as a submanifold in $\C^{n+2}$ with coordinates $u\in \C$, $v\in \C$, $z\in \C^n$  by the equation $uv = p(z)$, where the zero fiber of the polynomial $p \in \C[\C^n]$ is smooth (otherwise $X$ is not smooth)
 have ADP and thus DP
\cite{KaKu2}.

\medskip\noindent
(3) The only known non-algebraic examples with DP are firstly the holomorphic analogues of (2), namely the manifolds  $X$ given as a submanifold in $\C^{n+2}$ with coordinates $u\in \C$, $v\in \C$, $z\in \C^n$  by the equation $uv = f(z)$, where the zero fiber of the holomorphic function  $f \in \cO (\C^{n})$ is smooth (again otherwise $X$ is not smooth) \cite{KaKu2}. Secondly a special case of (4), namely when the Gizatullin surface can be completed by four rational curves, the Stein manifolds given by holomorphic  analogues of the concrete algebraic equations have the density property \cite{AKP}.

\medskip\noindent
(4) Smooth Gizatullin surfaces which admit a $\C$-fibration  with at most one singular  and reduced fibre. Sometimes these surfaces are called generalized Danielewski surfaces
\cite{Andrist}.

\medskip\noindent Recall that Gizatullin surfaces are by definition the normal affine surfaces on which the algebraic automorphism groups acts with an open orbit whose complement is a finite set of points. By the classical result of Gizatullin they can be 
characterized by admitting a completion with a simple normal crossing chain of rational curves at infinity. Every Gizatullin surface admits a $\C$-fibration  with at most one singular fibre which however
is not always reduced.

\medskip\noindent
(5) Certain hypersurfaces in $\C^{n+3}$ with coordinates $z=(z_0, z_1, \ldots , z_n) \in \C^{n+1}, x\in \C, y \in \C$ given by 
the equation $x^2y= a(z) + x b(z)$ where $\deg_{z_0} a \le 2$, $\deg_{z_0} b \le 1$ and not both degrees are zero,  including the Koras-Russell threefold from equation \eqref{KR} \cite{Leu}.

\medskip\noindent 
Here is a number of consequences the density property has, the proof of each of them is a certain application of the  Anders\'en-Lempert theorem:

\medskip\noindent
If $X$ is a Stein manifold with DP, then\\

\begin{enumerate}

\medskip
\item  $X$ is covered by Fatou-Bieberbach domains, i.e., each $x\in X$ has a neighborhood $\Omega_x \subset X$ biholomorphic to $\C^{\dim X}$ \cite{Varolin2}. In particular all Eisenman measures on $X$
vanish identically and bounded plurisubharmonic functions are constant.

\medskip
\item There is $\varphi : X \to X$, injective holomorphic not surjective (biholomorphic images of $X$ in itself) \cite{Varolin2}.

\medskip
\item If $X$ is Stein with DP,  $ \dim X \ge 3$ and $Y$ is a complex manifold such  that $End (X)$ and $End (Y)$ are isomorphic as abstract semigroups, then $X$ and $Y$ are biholomorphic or anti-biholomorphic, see \cite{Andrist1}, \cite{AndristWold}. We believe that the same is true if the dimension of $X$ is 2, but the known proofs do not apply.

\medskip
\item There are finitely many  complete holomorphic vector fields
$\theta_1, \ldots , \theta_N \in \CVF_{hol} (X)$ 
such that  ${ \rm span} (\theta_1 (x), \ldots , \theta_N (x)) = T_x X \quad  \forall x \in X$ (see \cite{KaKuPresent})
and thus $X$ is holomorphically flexible (see Definition \ref{holflex}).
\end{enumerate}

\subsection{Flexibility} 

\subsubsection{Definition and main features}

The notion of flexibility is even more  recent than the density property.  It was defined in \cite{AFKKZ}. First we state the algebraic version: 

\begin{Def}
Let $X$ be an  algebraic variety defined over $\C$ (any algebraically closed field would do). 
We let $\SAut(X)$ denote the subgroup of
$\Aut_{alg} (X)$ generated by all algebraic one-parameter unipotent
subgroups of $\Aut_{alg} (X) $, i.e., algebraic subgroups isomorphic to the
additive group $\G_a$ (usually denoted $\C^+$ in complex analysis). The group $\SAut(X)$ is called the
{\em special automorphism group} of $X$; this is a normal subgroup
of $\Aut_{alg} (X)$.
\end{Def}

\begin{Def}
We say that a point $x\in X_\reg$ is
{ \em algebraically   flexible} if the tangent space $T_x X$ is spanned by the
tangent vectors to the orbits $H\cdot x$ of one-parameter unipotent
subgroups $H\subseteq \Aut_{alg} (X)$. A variety $X$ is called {\em algebraically  flexible} if
every point $x\in X_\reg$ is.
\end{Def}

Clearly, $X$ is algebraically flexible if one
point of $X_\reg$ is and the group $\Aut_{alg} (X)$ acts transitively on
$X_{\rm reg}$.

The main feature of algebraic flexibility is the following result from \cite{AFKKZ} (whose proof mainly relies on the Rosenlicht theorem).
\begin{Thm}\label{mthm}
For an  irreducible affine variety $X$ of dimension $\ge 2$,
the following conditions are equivalent.
\begin{enumerate}
\item The group $\SAut (X)$ acts transitively on  $X_\reg$.
\item The group $\SAut (X)$ acts
infinitely transitively on $X_\reg$.
\item  $X$ is an algebraically  flexible variety.
\end{enumerate}
\end{Thm}

 The paper \cite{AFKKZ} also contains versions of simultaneous transitivity (where the space $X_\reg$ is stratified by orbits of  $\SAut (X)$) and versions with jet-interpolation. Moreover, it was recently remarked 
 that the theorem holds for quasi-affine varieties, see Theorem 1.11. in \cite{FKZ}.
 
 Examples of algebraically flexible varieties are homogeneous spaces of semisimple Lie groups (or extensions of semisimple Lie groups by unipotent radicals), toric varieties without non-constant invertible regular functions, cones over flag varieties and cones over Del Pezzo surfaces of degree at least $4$, normal hypersurfaces of the form
$uv = p(\bar x)$ in $\C_{u, v, \bar x}^{n+2}$. Moreover, algebraic subsets of codimension at least $2$ can be removed as recently shown by Flenner, Kaliman and Zaidenberg in \cite{FKZ}:

\begin{Thm}  Let $ X$ be a smooth quasi-affine variety of dimension $\ge 2$ and $Y \subset  X$ a closed subvariety of codimension $\ge  2$. If $X$ is flexible then so is $X\setminus Y$.

\end{Thm}

The holomorphic version  of this notion for a reduced complex space $X$ is much less explored, it is obviously implied by the algebraic version in case $X$ is an algebraic variety.

\begin{Def} \label{holflex} We say that a point $x\in X_\reg$ is
{\em holomorphically   flexible} if the tangent space $T_x X$ is spanned by the
tangent vectors of completely integrable holomorphic vector fields, i.e. holomorphic  one-parameter 
subgroups in $\Aut_{hol} (X)$. 
A reduced complex space  $X$ is called {\em holomorphically  flexible} if
every point $x\in X_\reg$ is.
\end{Def}
Clearly, $X$ is holomorphically flexible if one
point of $X_\reg$ is and the group $\Aut_{hol} (X)$ acts transitively on
$X_{\rm reg}$.

In the holomorphic category it is still open whether an analogue  of Theorem \ref{mthm} holds.

\smallskip\noindent
\textbf{Open Problem:} Are the three equivalences from Theorem  \ref{mthm}  true for a reduced  irreducible   Stein  space $X$?  More precisely, if a reduced irreducible  Stein  space $X$ is holomorphically flexible, does the holomorphic automorphism group 
$\Aut_{hol} (X)$ act infinitely  transitively on
$X_{\rm reg}$? 
\smallskip

It is clear that holomorphic flexibility of $X$ implies that $\Aut_{hol} (X)$ acts transitively on
$X_{\rm reg}$, i.e., the implication $ (3) \Rightarrow (1)$ is true.
Indeed, let $\theta_i, i=1, 2, \ldots, n$ be completely integrable holomorphic vector fields which span the tangent space $T_x X$ at some point $x \in X_{\rm reg}$, where $n= dim X$. If $\psi^i : \C \times X \to X, \quad (t, x)  \mapsto \psi^i_ t (x)$ denote the corresponding one-parameter subgroups, then
the map $\C^n \to X, \quad (t_1, t_2, \ldots, t_n) \mapsto \psi_{t_n}^n \circ \psi_{t_{n-1}}^{n-1} \circ \cdots \circ \psi_{t_1}^1 (x)$ is of full rank at $t =0$ and thus by the Inverse Function Theorem a local biholomorphisms from a neighborhood of $0$ to a neighborhood of $x$. Thus the $\Aut_{hol} (X)$-orbit
through any point of $X_{\rm reg}$ is open. If all orbits are open, each orbit is  also closed, being the complement of all other orbits. Since $X_{\rm reg}$ is connected, this implies that it consists of one orbit.

The inverse implication  $ (1) \Rightarrow (3)$ is also true. For the proof we appeal to the Hermann--Nagano Theorem which states that if $\mathfrak g$ is a Lie algebra of holomorphic vector fields on a manifold
$X$, then  the orbit $R_\g  (x) $ (which is the union of all points $z$ over any collection of finitely many fields $v_1, \ldots v_N \in \g $  and over all times $(t_1, \ldots, t_N)$ for  which the expression $z =\psi_{t_N}^N \circ \psi_{t_{N-1}}^{N-1} \circ \cdots \circ \psi_{t_1}^1 (x)$ is defined)
  is a locally closed submanifold and its tangent space at any point $y \in R_\g  (x)$ is $T_y R_\g (x) = span_ { v \in \g} {v (y)}$. We consider the Lie algebra $\g = \Lie (\CVF_{hol} (X))$
generated by completely integrable holomorphic vector fields. Since by the assumption the orbit is $X_{\rm reg}$, we conclude that Lie combinations of completely integrable holomorphic vector fields span the tangent space at each point in  $X_{\rm reg}$. Now suppose at some point $x_0$ the completely integrable fields do not generate $T_{x_0} X_{\rm reg}$, i.e., there is a proper linear subspace $W$ of $T_{x_0} X_{\rm reg}$, such that $v(x_0) \in W$ for all completely integrable  holomorphic fields $v$. Any Lie combination of completely integrable  holomorphic fields is a limit (in the compact open topology) of sums of completely integrable  holomorphic fields due to the formula $\{ v, w\} = \lim_{t \to 0} \frac {\phi_t^* (w) - w} t$, where $\phi_t$ is the flow of $v$, for the Lie bracket ($\phi_t^* (w)$ is a completely integrable field pulled back by an automorphism, thus completely integrable!). Therefore all Lie combinations of completely integrable fields evaluated  at  $x_0$ are contained in $W \subset T_{x_0} X_{\rm reg}$, a contradiction.

In order to prove the  remaining implication $(3) \Rightarrow (2)$ in the same way as in the algebraic case, one would like  to find suitable  functions $f \in {\ker } \  \theta$ for a completely integrable holomorphic vector field $\theta$, vanishing at one point and not vanishing at some other point of $X$.  In general these functions may not exist, an orbit of $\theta$ can be holomorphically Zariski dense in $X$.

At this point it is worth mentioning that for a Stein manifold the density property DP implies all three conditions from Theorem \ref{mthm}. For flexibility this has been mentioned above,  infinite transitivity (with jet-interpolation) is proved by Varolin in \cite{Varolin2}, see also Lemma \ref{movingpoints} below.

The main importance of holomorphic flexibility is the fact that holomorphically flexible manifolds are sources for the Oka-Grauert-Gromov-h(omotopy)-principle in Complex Analysis. 

Let us explain this more precisely. A holomorphically flexible complex manifold $X$   is an Oka--Forstneri\v c manifold which means it is an appropriate (nonlinear) target for generalizing classical Oka--Weil interpolation and Runge approximation for holomorphic functions (linear target $\C$) or sections of vector bundles (linear target as well). More precisely,  the following is true for an Oka--Forstneri\v c manifold  $X$ (see \cite{For} Corollary 5.4.5.): 

\medskip\noindent
 {\bf Oka  principle}\label{Oka}

{\em 
\smallskip\noindent
For any Stein space $W$, complex subspace $W^\prime$, compact $\mathcal O (W)$-convex subset  $K = \widehat{K} \subset W$
and any $\varphi :  W \to X$ continuous, such that the restriction to $W^\prime  \cup K$ is holomorphic,
there is a homotopy of continuous maps

$$ h: [0,1] \times W \to X $$

\noindent
from the continuous  map $h_0 = \varphi$ to  a holomorphic map $h_1$,   \\

\smallskip\noindent
with interpolation:\hskip 1.4  cm  $h_t = \varphi$ on $W^\prime$ 
and \\
with approximation:  \hskip 1 cm  $\vert h_t - \varphi \vert_K$ arbitrary small $\ \forall \ t\in [0,1]$.\\

\noindent
Moreover parametric versions are true: The inclusion of the space of holomorphic maps $\Hol (W, X)$ into the space of continuous maps $\Cont (W, X)$ is a weak homotopy equivalence. 

}

We refer the reader to the monograph of Forstneri\v c for more details. Let us just remark that Gromov 
introduced the notion of
an elliptic manifold, which by definition is a complex manifold with a dominating spray, and proved 
that the above conclusions are true for an elliptic manifold.

 A spray for a complex manifold $X$ is a holomorphic vector bundle $\pi : E \to X$ together with a holomorphic (spray) map
$s: E \to X$, such that $s$ is the identity on the zero section $X\hookrightarrow E$. The spray is dominating 
if for each $x \in X$ the induced differential map  sends the fibre $E_x = \pi^{-1} (x)$, viewed as a linear subspace of $T_x E$ surjectively onto $T_xX$. Gromov's standard example for a spray is used to see that
a holomorphic flexible manifold is elliptic. 

\begin{Exa}\label{Gromovspray} Let $X$ be holomorphically flexible. We need an easy fact proved by the author in \cite{flexibility} (see also the appendix of \cite{AFKKZ}), namely that there are finitely many completely integrable holomorphic vector fields $\theta_1, \theta_2, \ldots, \theta_N$ such that at each point $x\in X$ they
span the tangent space (by definition  for each point there are finitely many spanning at this point only). 
Let $\psi^i : \C \times X \to X, (t,x) \mapsto \psi_t^i (x)$ denote the corresponding flow maps. 

Then the map $s: \C^N \times X \to X$
defined by $((t_1, t_2, \ldots, t_N), x) \mapsto \psi_{t_N}^N \circ \psi_{t_{N-1}}^{N-1} \circ \cdots \circ \psi_{t_1}^1 (x)$ is of full rank at $t=0$ for any $x$,
and thus a dominating spray map from the trivial bundle $X\times \C^N \to X$.
\end{Exa}

\section{Applications to natural geometric questions}

All applications described below are applications of the density property DP and some of them use flexibility at the same time. Sometimes variants of the density property 
are used. There are versions for volume preserving automorphisms, so called Volume Density Property, relative versions, e.g., for automorphisms fixing subvarieties or 
fibered versions (considering automorphisms  leaving invariant a fibration). We leave it to the interested reader to look up in the given references which version is used. Only the 
very last application does not use DP, it uses flexibility only via a stratified version of the Oka principle.

\smallskip\noindent
{\bf (1)}
A first application that we would like to mention is to the notoriously difficult question whether every open Riemann surface can be properly holomorphically embedded into $\C^2$. This is the only dimension for which   the conjecture of Forster \cite{Fo}, saying
that every Stein manifold of dimension $n$ can be properly holomorphically embedded into $\C^N$ for $N= [\frac{3 n}{2}] + 1$, is still unsolved. The conjectured dimension is sharp by examples of Forster \cite{Fo} and has been proven by Eliashberg, Gromov \cite{EG} and Sch\"urmann \cite{Sch} for all dimensions $n\ge 2$.
Their methods of proof fail in dimension $n=1$. But Forn\ae ss Wold invented a clever combination of a use of shears (nice  projection property) and Theorem \ref{AL-Theorem} which led to  many new embedding theorems for open Riemann surfaces. As an example
we like to mention the following two recent results of Forstneri\v c and Forn\ae ss Wold \cite{FW}, \cite{FW1} the first of them being the most general one for open subsets of the complex line:

\begin{Thm} Every domain in the Riemann sphere with at least one  and at most countably many boundary components, none of which are points, admits a proper holomorphic embedding into $\C^2$.
\end{Thm}

\begin{Thm}  If $\bar \Sigma$ is a (possibly reducible) compact complex curve in $\C^2$ with boundary $\partial \Sigma$ of class $C^r$ for some $ r > 1$, then the inclusion map 
$i: \Sigma = \bar \Sigma\setminus  \partial \Sigma \to \C^2$ can be approximated, uniformly on compacts in $\Sigma$, by proper holomorphic embeddings $\Sigma \to \C^2$. 
\end{Thm}

Many versions of embeddings with interpolation are also known and proven using the same  methods invented by Forn\ae ss Wold in \cite{W}. In particular the Gromov--Eliashberg--Sch\"urmann Theorem mentioned above is true with interpolation on discrete subsets \cite{IFKP}.

\smallskip\noindent
{\bf (2)}
Another application is to construct non-straightenable holomorphic embeddings of $\C^k$ into $\C^n$ for all pairs of dimensions $0<k<n$, a fact which is contrary to the situation in affine algebraic geometry. It is contrary to the famous Abhyankar-Moh-Suzuki theorem for $k=1, n=2$ and also to work of Kaliman \cite{Ka} for  $2k+1 < n$, whereas straightenability for the other dimension pairs is  still unknown in algebraic geometry. Here non-straightenable for an embedding  $\C^k$ into $\C^n$ means to be not equivalent to the standard embedding.

To give the reader an idea about the flavor of the subject, let us explain briefly how Theorem \ref{AL-Theorem} is used to construct a non-straightenable embedding of
$\C$ into $\C^n$.

The idea is to use the existence of a non-tame discrete subset $E=\{ e_1, e_2, \ldots \}$ in $\C^n$ and to construct
a proper holomorphic embedding $\varphi : \C \to \C^n$ whose image contains $E$, $\varphi (\C)\supset E$. The notion of a tame subset in $\C^n$ goes back to 
Rosay and Rudin \cite{RR88} and is by definition a subset in $A \subset \C^n$ which can be mapped by a holomorphic automorphism $\alpha \in \Aut_{hol} (\C^n)$ onto the
subset $\{ (i, 0 \ldots, 0)\in \C^n  \ , i \in \N \}$.
The only information we need is the fact that any countable discrete subset of the first
coordinate line $\{ (z, 0 \ldots, 0)\in \C^n  \  : \ z \in \C \}$ is tame. This shows that
our embedding containing a non-tame subset $E$ cannot be straightenable. Indeed, were it straightenable, the set $E$ would be mapped by the straightening automorphism into the first coordinate line and therefore be tame, a contradiction.
The existence of non-tame subsets of $\C^n$ has been proved by Rosay and Rudin \cite{RR88} and for the reader familiar with Eisenman measures we just mention that
one can construct a discrete subset in $\C^n$ whose complement is Eisenman $n$-hyperbolic (also called volume hyperbolic). Since the complement of the first coordinate line has constantly vanishing Eisenman volume such an $E$ cannot be contained in a coordinate line. The step where we embed $\C$ through that set is where we make use of DP applying Theorem \ref{AL-Theorem}. Let us describe this a little more detailed in order to let the reader feel the flavor of the subject. The step where we embed the line through a discrete subset is presented in the more general situation where $\C^n$ is replaced by any Stein manifold with DP.

We begin with an easy lemma which already demonstrates the power of the density property. Expressed in words it means that $N$ different   points can be moved around independently of each other by automorphisms   inside small neighborhoods, and the nearer the points are to their targets the nearer the automorphism is to the identity.
This lemma easily implies that the holomorphic automorphism group of $X$ acts infinitely transitive on $X$, a result of Varolin \cite{Varolin2}.  

\begin{Lem}\label{movingpoints} For any $N$-tuple of pairwise distinct points $x_1, x_2, \ldots, x_N$ in a Stein manifold $X$ of dimension $\dim X = n$ with DP, there is an open neighborhood  $P$ of $0 \in \C^{n\cdot N}$ together with an injective holomorphic map
$\psi : P \to \Aut_{hol} (X)$ satisfying $\psi (0) = id$ such that the map $ P \to X^N$ defined by $p \mapsto (\psi (p) (x_1),
 \psi (p) (x_2), \ldots , \psi (p) (x_N)$ is a biholomorphism from $P$ to an open neighborhood of the point 
 $(x_1, x_2, \ldots , x_N)$ in the $N$-fold product $X^N$.
\end{Lem}
\begin{proof} Choose coordinate balls $B_i$ which are Runge around each point $x_i$ so small that their union $\cup_{i=1}^N B_i$ is Runge too.  Consider the collection of $n N$ vector fields $\theta_i^j$ defined on $\cup_{i=1}^N B_i$ by 
$\theta_i^j =  \frac \partial {\partial z_j}$ on $B_i$ and identically zero on $B_k$ for $k\ne i$. They naturally induce
vector fields on $B_1 \times B_2 \times \cdots \times B_N$ which span the tangent space to $X^N$ at each point there.
Using DP they can  be approximated by Lie combinations, in fact sums (see the proof of $(1) \Rightarrow (3)$ after Definition \ref{holflex}) of complete vector fields. Linear algebra shows that there are $nN$ complete holomorphic
vector fields $\tilde \theta_i$ whose naturally induced fields  on $X^N$ span the tangent space at the point $(x_1, x_2, \ldots, x_N)$. The map from $\C^{nN}$ to $X^N$ defined  by applying the flows of the fields in any (but fixed) order to the point 
$(x_1, x_2, \ldots, x_N)$  (as in Example \ref{Gromovspray}) is of full rank. By the implicit function theorem it is a local biholomorphism, which finishes the proof.
	\end{proof}

\begin{Prop} \label{interpol} Given an analytic subset $S$ in a Stein manifold $X$ with DP and a countable discrete subset $E\subset X$. Then there is a proper holomorphic embedding
$\varphi : S \hookrightarrow X$ with $E \subset \varphi (S)$.
\end{Prop}
\begin{proof} The idea is to construct inductively a sequence of holomorphic automorphisms $\alpha_i \in \Aut_{hol} (X)$ such that the limit $\lim_{n \to \infty}
\alpha_n \circ \cdots \circ \alpha_2 \circ \alpha_1$ converges uniformly on compacts
on an open subset $\Omega$ of $X$ containing $S$	 to a biholomorphic map $\psi : \Omega \to X$ (thus $\Omega$ is a so called Fatou-Bieberbach domain in $X$). Moreover $E \subset\psi (S)$. For describing the inductive procedure we fix a strictly plurisubharmonic (spsh)
exhaustion function $\rho: X \to [0, \infty)$ of the Stein manifold $X$. The existence
of such a function is guaranteed by embedding $X$ as a closed submanifold into some 
affine space $\C^N$ of high enough dimension and restricting the function $\vert z \vert^2 = \vert z_1 \vert^2+\vert z_2 \vert^2 + \ldots + \vert z_N \vert^2$ to $X$.
By moving the origin by an arbitrarily small amount in $\C^N$ if necessary and changing the enumeration of the points in $E$, we can assume that 
$\rho (e_1) < \rho (e_2) < \ldots \rho (e_k) < \rho (e_{k+1}) < \ldots$. Also we choose
numbers $r_k >0$ with $\rho (e_k) < r_k <\rho (e_{k+1})$. Since $\rho$ is spsh its sub-level sets $X_r := \rho^{-1} ([0, r))$ are Runge subsets of $X$. We can
restrict $\rho$ to the closed subvariety $S$ in order to obtain a spsh exhaustion function of $S$. Again the sub level sets $S_R := \rho^{-1} ([0, R) \cap S$ are 
holomorphically convex Runge subsets of $S$.
Moreover 
it can be shown that for each  $R>r$ there is a neighborhood basis $U_i$ of $S_R$ with the property that $X_r \cup U_i$ is a Runge subset in $x$.
Finally choose $\epsilon_k < r_{k+1} - r_k$ and with $\sum_k \epsilon_k < \infty$.

We can assume that $e_1 \subset S$ by the following application of Theorem \ref{AL-Theorem}: Suppose $e_1 \notin S$, choose a point $ s \in S$ and connect the points 
$s$ and $e_1$ by a continuous path $\gamma (t)$. It is an easy exercise  using local charts and cutoff functions that the map $\gamma_t$
can be extended to a neighborhood $W$ of the point  $s$ to  a map
$\Gamma : W \times [0,1] \to X$ with $\Gamma (s, t) = \gamma (t)$ so that the maps $\Gamma_t$ are biholomorphisms of
$W$ onto its image $\Gamma_t (W)$ which is Runge in $X$ and with $\Gamma_0 = id$. For 
$X=\C^n$ a simple translation $\Gamma_t (z) = z + \gamma (t)$ will do the job.
By Theorem \ref{AL-Theorem} we can approximate the final map $\Gamma_1$ by an
automorphism which will move the point $e_1$ arbitrarily close to the point $s \in S$. Using Lemma \ref{movingpoints} we can move $e_1$ by a further automorphism exactly to $s \in S$.

Now we describe the inductive step: 
The automorphisms $\alpha_k$
will satisfy 

$$  \Vert \alpha_k -id \Vert_{X_{r_k} \cup S_{r_{k+1}}} < \epsilon_k $$
$$ \alpha_k (e_i ) = e_i \  \ \ i = 1, 2, \ldots, k$$
$$  \exists s \in S \ : \ \alpha_k \circ \ldots \circ  \alpha_1 (s) = e_{k+1}$$

We assume $\alpha_1, \ldots \alpha_{k-1}$ have been constructed. To construct $\alpha_k$ we proceed as follows: If 
$  e_{k+1} \in \alpha_{k-1} \circ \ldots \circ  \alpha_1 (S)$ we set $\alpha_k = id$, if not we connect  $e_{k+1}$ by a 
continuous path $\gamma (t)$ never intersecting $X_{r_k} \cup S_{r_{k+1}}$ to a point $s$ in $S \setminus S_{r_{k+1}}$.
Now the local data for application of Theorem \ref{AL-Theorem} is identity on $X_{r_k} \cup S_{r_{k+1}}$ and an extension
of the path $\gamma (t)$ to biholomorphic maps $\Gamma_t$ of a small neighborhood $W$ of $s$ as above. If $W$ is small enough the
sets $X_{r_k} \cup S_{r_{k+1}}\cup \Gamma_t (W)$ are Runge and an application of Theorem \ref{AL-Theorem} gives an automorphism
$\alpha$ which satisfies the first of the three conditions above and the second and third condition are satisfied approximately. By composing $\alpha$ with an automorphism from Lemma \ref{movingpoints} (not destroying the first condition, if the points had moved/stayed  nearby enough) we find our desired $\alpha_k$.
The claim that the limit $\lim_{n \to \infty}
\alpha_n \circ \cdots \circ \alpha_2 \circ \alpha_1$ converges uniformly on compacts
on an open subset $\Omega$ of $X$ containing $S$	 to a biholomorphic (Fatou-Bieberbach) map $\psi : \Omega \to X$ follows standardly from $  \Vert \alpha_k -id \Vert_{X_{r_k}}  < \epsilon_k $ and the restrictions on the $\epsilon_k$ (see e.g. \cite{For} Proposition 4.4.1. and Corollary 4.4.2.). Moreover $  \Vert \alpha_k -id \Vert_{S_{r_{k+1}}} < \epsilon_k $ implies $S \subset \Omega$ and since
$S$ is closed in $\Omega$ clearly $\psi(S)$ is closed in $X=\psi (\Omega)$ and the last two properties of $\alpha_k$ ensure that $E\subset \psi (S)$.
\end{proof}

More generally one can think of the possible ways of embedding any Stein manifold $X$ into $\C^n$.

\begin{Def}\label{def-eq-emb}
Two embeddings $\Phi,\Psi\colon X\hookrightarrow\C^n$ are {\it equivalent} if there 
exist automorphisms $\varphi\in\Aut(\C^n)$ and $\psi\in\Aut(X)$ such that 
$\varphi\circ\Phi=\Psi\circ\psi$.
\end{Def}

The best  and quite striking  result in this direction says that there are even holomorphic families of pairwise non-equivalent  holomorphic embeddings.

\begin{Thm}  \cite{Kutzschebauch-Lodin, KutzschebauchBorell}.
  \label{mainembedding} Let $n, l$ be natural numbers with $n\ge l+2$.
There exist, for $k=n-l-1$,  a family of holomorphic embeddings of
$\C^l$ into $\C^n$ parametrized by $\C^k$, such that for
different parameters $w_1\neq w_2\in \C^k$ the embeddings
$\psi_{w_1},\psi_{w_2}:\C^l \hookrightarrow \C^{n}$ are
non-equivalent. Moreover, there are uncountably many non-equivalent holomorphic embeddings of
$\C^{n-1}$ into $\C^n$. 
\end{Thm} 

We will see a beautiful    application of Theorem \ref{mainembedding} to the holomorphic linearization problem in the last section.

It is clear that  in Definition \ref{def-eq-emb} the ambient space $\C^n$ can be replaced by any other manifold, most
interesting by a Stein manifold with DP, since they are  all targets for embedding of Stein manifolds, exactly as affine spaces are, see the next application.

\medskip\noindent
\textbf{Open Problem:}   \label{equiv} Suppose $X$ is a Stein manifold with density property and $Y \subset X$ is a closed submanifold. Is there always another proper holomorphic embedding $\varphi : Y \hookrightarrow X$ which is not equivalent to
the inclusion $ i: Y \hookrightarrow X$?
\smallskip

We should remark that an affirmative answer  to this problem is stated in \cite{Varolin2}, but the author apparently had another (weaker) notion of equivalence in mind.

\smallskip\noindent
{\bf (3)}
As mentioned above not only affine spaces $\C^n$ are the universal targets for embedding Stein manifolds. All Stein manifolds with DP are such targets as well.

\begin{Thm} \cite{AFRW}
 Let $X$ be a Stein manifold satisfying the density property. If $S$ is a Stein manifold and $2 \dim S + 1 \le  \dim X$, then any continuous map $f : S \to  X$ is homotopic to a proper holomorphic embedding $F : S ֒\to X$. If in addition $K$ is a compact $\cO (S)$-convex set in $S$ such that $f$ is holomorphic on a neighborhood of $K$, and $S^\prime$ is a closed complex subvariety of $S$ such that the restricted
map $f : S^\prime  ֒\to  X$ is a proper holomorphic embedding of $S^\prime $ to $X$, then $F$ can be chosen to agree with $f$ on $S^\prime $ and to approximate $f$ uniformly on $K$ as closely as desired.
 \end{Thm}

This theorem which underlines the intrinsic importance of
Stein manifolds with DP could play a future role for 
results in the direction of 
L\'arussons question, which is motivated by
his homotopy theoretic view point of Oka theory, for details we refer to \cite{Larusson}.

\medskip\noindent
\textbf{Open Problem:} Does every Stein manifold admit an acyclic proper holomorphic embedding into a Stein Oka--Forstneri\v c manifold?

For the connection to the density property remember that Stein manifolds with DP are Oka--Forstneri\v c manifolds.

\smallskip\noindent
{\bf (4)}
The already mentioned infinite transitivity of the action of 
$\Aut_{hol} (X)$ on a Stein manifold 
with density property can be strengthened to a parametric version. This is a recent result of Ramos Peon and the author.
For an interpretation of it as an Oka 
principle in the Grauert style but for bundles with infinite
dimensional Fr\' echet groups as fibers  (instead of Lie groups)
we refer the interested reader to the last section of their 
paper \cite{Kutzschebauch-RamosPeon}.

\begin{Thm}
Let $W$ be a Stein manifold and $X$ a Stein manifold with the density property. \\
Let $x:W\to X^N\setminus \{(z^{1},\dots,z^{N}) \in X^N  : z^{i} = z^{j} \text{ for some } i  \ne j \}$ \\
be a holomorphic map. Then the parametrized points $x^{1} (w) ,\dots,x^{N}(w)$ are simultaneously standardizable by an automorphism 
lying in the path-connected component of the identity $(\Aut_W(X))^0$ of $\Aut_W(X)$ if and only if $x$ is null-homotopic.\\
\end{Thm}
\noindent
Here simultaneously standardizable means that given any fixed positions $\tilde x_1, \tilde x_2,$ $ \ldots , \tilde x_N \in X$ there are holomorphic automorphisms $\alpha$ of $X$ depending holomorphically on $w$, i.e., an element of $\Aut_{W}(X)=\{ \alpha\in\Aut(W\times X); \alpha(w,z)=(w,\alpha^w(z))\},$
with 
$\alpha^w(x^j(w))=\tilde x_j$
for all $w\in W$ and $j=1,\dots,N$.

The proof of this result uses extensively DP via Theorem \ref{AL-Theorem} and the 
Oka principle.
 
\smallskip\noindent
{\bf (5)}
All Fatou--Bieberbach domains arising as basins of attraction 
or more generally as domains of convergence of sequences of automorphisms of $\C^n$ are always Runge domains.
Thus it is natural to ask whether all Fatou--Bieberbach domains in $\C^n$ have to be Runge.
This problem was solved by {Forn\ae ss Wold} who constructed a Fatou--Bieberbach domain in $\C \times \C^*$ which is not Runge in $\C^2$
(but Runge in $\C \times \C^*$) using the density property of $\C \times \C^*$ \cite{W1}.

\smallskip\noindent
{\bf (6)}
One of the questions coming from complex dynamical systems is the description of the boundaries of Fatou--Bieberbach domains.
Say, a surprising result of {Stens{\o}nes} \cite{S}) provides such a domain in $\C^2$ with a smooth boundary which has, therefore, Hausdorff dimension $d=3$.
Furthermore, it was established by methods of complex dynamical systems that such a dimension can take any value $3 \leq d <4$.
However the question about a Fatou--Bieberbach domain in $\C^2$ with a boundary of Hausdorff dimension $d=4$ remained open until
{Peters} and {Forn\ae ss Wold} \cite{PW} managed to construct it using  the DP.

\smallskip\noindent
{\bf (7)}
 A question  posed by  {Siu} asks whether there exists always a Fatou--Bieberbach domain contained in the complement
to a closed algebraic subvariety $Z$ of $\C^n$ such that $\dim Z \leq n-2$.
The affirmative answer was obtained by {Buzzard} and {Hubbard} \cite{BH} who used some concrete construction.
Another proof of this fact was given by Kaliman and the author who used a version of the density property for such complements.
More precisely for any point $x \in \C^n\setminus Z$ there is a Fatou--Bieberbach (i.e. holomorphic injective) map
 $ f: \C^n \to \C^n\setminus Z$ with $f(0) =x$ (actually,  the existence of complete holomorphic vector fields on such complements   
 $\C^m\setminus Z$ has been observed in earlier papers of {Gromov} \cite{GromovOPHSEB} and {Winkelmann} \cite{Win}).

In particular all Eisenman measures on $\C^n \setminus Z$ are trivial.
It is worth mentioning that closed {\bf analytic} subsets of $\C^n$ of codimension $k$ may have $k$-Eisenman
hyperbolic complements. More precisely,  it was shown in \cite{BK} that
if a complex manifold $Y$ admits a proper holomorphic embedding into $\C^n$ then it has also another
proper holomorphic embedding with $(n- \dim Y)$ -Eisenman hyperbolic complement to the image (the proof is based on the  DP
and a generalized idea from \cite{BFo}).

\smallskip\noindent
{\bf (8)}
A beautiful combination of differential-topological methods with hard analysis (solutions of $\bar\partial$-equations with exact estimates) and the
Anders\'en-Lempert-Theorem is required for understanding of
how many totally real differentiable  embeddings of a real manifold $M$ into $\C^n$ can exist.

If $f_0,f_1\colon M\to{\bf C}^n$ are two totally real, polynomially convex real-analytic embeddings of a compact manifold
$M$ into ${\C}^n$, we say that $f_0$ and $f_1$ are $\Aut_{\rm hol} (\C^n)$-equivalent\footnote{It is unfortunate that
in the literature the term   ``$\Aut_{\rm hol} (\C^n)$-equivalence" is used in different meanings - see the sentence after the Open Problem in application (2).} if $f_1=F\circ f_0$, where
$F\colon U\to F(U)\subset{\bf C}^n$ is a biholomorphism defined in a neighborhood $U$ of $f_0(M)$ such that $F$ is the uniform limit in $U$
of a sequence of elements of $\Aut_{\rm hol}  (\C^n)$. Conditions for $\Aut_{\rm hol}  (\C^n)$-equivalence were found in \cite{FR},
using volume-preserving automorphisms  (and an approach using automorphisms preserving the holomorphic symplectic form
was considered in \cite{FSymp}).

In the smooth case let  ${\Eeul}^r(M,{\bf C}^n)$ be the set of all totally real polynomially convex $C^r$-embeddings of $M$ into ${\C}^n$
(for $2\le r\le\infty$). It is proved  by   {Forstneri\v c} and   {L{\o}w} that two embeddings $f_0, f_1\in{\Eeul}^\infty(M, {\C}^n)$ belong
to the same connected component (in the space of $C^r$-embeddings of $M$ into $\C^n$ equipped with the usual topology
of uniform convergence of all derivatives up to order $r$) if and only if there exists a sequence $\{\Phi_j\}\subset{\Aut_{\rm hol} }({\C}^n)$ such that
$\Phi_j\circ f_0\to f_1$ and $\Phi^{-1}_j\circ f_1\to f_0$ in $C^\infty(M)$ as $j\to \infty$. Precise results in the case $r<\infty$ were obtained in  \cite{FLO}.

\smallskip\noindent
{\bf (9)}
Recall that a ``long $\C^n$" is a complex manifold $X$
that can be exhausted by open subsets $\Omega_i$ which are all biholomorphic to $\C^n$,
i.e.  $X = \bigcup_{i=1}^\infty \Omega_i$, $\Omega_i \subset \Omega_{i+1} $, and $\Omega_i \cong \C^n$  for all $ i \in \N$.
(Here, of course,
$\Omega_i \subset \Omega_{i+1}$ is not a Runge pair.) The fist examples of such manifolds not biholomorphic
to $\C^n$  for $n \geq 3$ were constructed by Forn\ae ss  \cite{Forn} in 1976. The case of $n=2$ had been
resistant until recently when, developing further the ideas from application (5),
  {Forn\ae ss Wold} \cite{W0} constructed a ``long $\C^2$" which is not
biholomorphic to $\C^2$. Every known ``long $\C^n$" (including the case of $n=2$) is non-Stein. Recently families of non-isomorphic 
long $\C^n$'s have been constructed using DP via Theorem \ref{AL-Theorem} by Forstneri\v c and Boc Thaler \cite{FoBoc}.

\smallskip\noindent
{\bf (10)}
The classical approximation theorem of   {Carleman} states that
for each continuous function $\lambda : \R \to \C$ and a positive continuous function
$\epsilon : \R \to (0, \infty)$ there exists an entire function $f$ on $\C$ such that $\vert f(t) - \lambda (t)\vert < \epsilon (t)$ for every $t\in \R$.

Using the   {Anders\'en-Lempert} theorem together with some explicit shears 
  {Buzzard} and   {Forstneri\v c} \cite{BF} were able to prove a similar result for holomorphic embeddings into $\C^n$.
Namely, for any proper embedding $\lambda : \R \to \C^n$ of class $C^r$ (where $n\geq 2$ and $r\ge 0$)
and a positive continuous function $\epsilon : \R \to (0, \infty)$ there exists a proper holomorphic embedding
$f: \C \to \C^n$ such that $$ \vert f^{(s)} (t) - \lambda^{(s)} (t)\vert < \epsilon (t) \quad \forall \ t\in \R, \ 0\le s\le r.$$

Actually this fact remains valid under the additional requirement that the embedding satisfies the interpolation property as in  Proposition \ref{interpol}.

\smallskip\noindent
{\bf (11)} The \emph{spectral ball} of dimension $n \in \N$ is defined to be
\[
\Omega_n := \{ A \in \mat{n}{n}{\C} \;:\; \rho(A) < 1 \}
\]
where $\rho$ denotes the spectral radius, i.e.\ the modulus of the largest eigenvalue.

The study of the group of holomorphic automorphisms of the spectral ball started with the work of Ransford and White \cite{RanWhi1991} in 1991 and was continued by various authors. A conjecture from \cite{RanWhi1991} was disproved by Kosi{\'n}ski \cite{Kosinski2013} for $n=2$ and moreover he described a dense subgroup of the $2 \times 2$ spectral ball $\Omega_2$. Andrist and the author  generalized these  results to $\Omega_n$ for $n \geq 2$ with an new approach using the fibered density property. The notion of $\Sl_n(\C)$-shears and -overshears is easy to understand. One uses a one-parameter subgroup of $\Sl_n (\C)$ which by conjugation acts on $\Omega_n$ and thus gives rise to a complete vector field $\theta$ on $\Omega_n$. Multiplying $\theta$
 with a function $f$ in the kernel gives a shear field $f \theta$ and with a function $f$  in the second kernel ($\theta (f) \in \ker f \setminus \{0\}$) we get an overshear field $f \theta$. 
 Shears and overshears are time $t$-maps of the corresponding fields.
 This is the same way as  shears and overshears in $\C^n$ from formulas \ref{shear} and \ref{overshear} arise from the complete fields $\partial \over {\partial z_i}$. 

\begin{Thm} \cite{AndKu}
The $\Sl_n(\C)$-shears and the $\Sl_n(\C)$-overshears together with matrix transposition and matrix valued M\"obius transformations generate a dense subgroup (in compact-open topology) of the holomorphic automorphism group $\Aut{\Omega_n}$.
\end{Thm}

For the convenience of the reader we recall that a matrix valued M\"obius transformation is a map of the form
$h \colon \Omega_n \to \Omega_n$
\begin{equation}
\label{eqmoebius}
A \mapsto \gamma \cdot (A - \alpha \cdot \id) \cdot (\id - \overline{\alpha} A )^{-1}, \quad \alpha \in \udisk, \gamma \in \boundary \udisk
\end{equation}

\smallskip\noindent
{\bf (12)} It is standard material in a Linear Algebra course that the group $\mbox{SL}_m(\mathbb{C})$ is
generated by elementary matrices $E+ \alpha e_{ij}, \ i\ne j$, i.e., matrices with 1's on the diagonal and all entries
outside the diagonal are zero, except one entry.  Equivalently, every matrix $A \in \mbox{SL}_m(\mathbb{C})$
can be written as a finite product of upper and lower diagonal unipotent matrices (in interchanging order). The same question for matrices in $\mbox{SL}_m(R)$ where $R$ is a commutative ring 
instead of the field $\mathbb{C}$ is much more delicate. For example, if $R$ is the ring of
complex valued functions (continuous, smooth, algebraic or holomorphic)  from a space $X$ the problem amounts to finding for a given map $f : X \to \mbox{SL}_m(\mathbb{C})$ a factorization as a product of upper and lower diagonal unipotent matrices
\begin{equation*}f(x) = \left(\begin{matrix} 1 & 0 \cr G_1(x) & 1 \cr \end{matrix} \right)   
\left(\begin{matrix} 1 & G_2(x) \cr 0 & 1 \cr \end{matrix} \right)  \ldots \left(\begin{matrix} 1 & G_N(x)\cr 0 & 1 \cr \end{matrix} \right) 
\end{equation*}
where the $G_i$ are maps $G_i : X \to \mathbb{C}^{m(m-1)/2}$.

Since any product of (upper and lower diagonal) unipotent matrices is homotopic to a constant map (multiplying each entry outside the diagonals by $ t \in [0, 1]$ we get a homotopy to the identity matrix), one has to assume that the given map $f : X \to \mbox{SL}_m(\mathbb{C})$ is homotopic to a constant map or as we will say {\it null-homotopic}. In particular this assumption holds if the space $X$ is contractible.

This very general problem has been studied in the case of polynomials of $n$ variables. For $n=1$, i.e.,
$f : \mathbb{C} \to \mbox{SL}_m(\mathbb{C})$ a polynomial map (the ring $R$ equals $\mathbb{C}[z]$) it is an easy consequence of the fact that $\mathbb{C}[z]$ is an Euclidean ring  that such $f$ factors through a product of upper and lower diagonal unipotent matrices. For $m=n=2$ the following
counterexample was found by   {Cohn} \cite{CohnSGL2R}: the matrix
 $$\left(\begin{matrix} 1-z_1z_2 &  z_1^2\\ -z_2^2 & 1+z_1z_2 \end{matrix}\right)\in \mbox{SL}_2(\mathbb{C}[z_1,z_2]) $$ does not decompose as a finite product of unipotent matrices.

For $m\ge 3$ (and any $n$) it is a deep result of   {Suslin} \cite{SuslinSSLGRP} that any matrix in $\mbox{SL}_m(\mathbb{C}[\mathbb{C}^n])$ decomposes as a finite product of unipotent (and equivalently elementary) matrices. 
More results in the algebraic setting can be found in \cite{SuslinSSLGRP}  and \cite{GrunewaldGSL2}.
For a connection to the Jacobian problem on $\mathbb{C}^2$ see \cite{WrightAFPS}.

In the case of continuous complex valued functions on a topological space $X$ the problem was
studied and partially solved  by   {Thurston} and   {Vaserstein} \cite{ThurstonVasersteinK1ES} and  then finally solved by   {Vaserstein} \cite{VasersteinRMDPDFAO}.

It is natural to consider the problem for rings of holomorphic functions on Stein spaces, in particular on
$\mathbb{C}^n$. Explicitly this problem was posed by   {Gromov} in his groundbreaking paper \cite{GromovOPHSEB}  where he extends the classical   {Oka-Grauert} theorem from bundles with homogeneous fibers to fibrations with elliptic fibers, e.g., fibrations admitting a dominating spray (for definition see before Example \ref{Gromovspray}).
In spite of the above mentioned result of   {Vaserstein} he calls it the 

\medskip\noindent
{\bf {Vaserstein problem:}} (see \cite[sec 3.5.G]{GromovOPHSEB}) 

{\sl Does every holomorphic map $\mathbb{C}^n \to \mbox{SL}_m(\mathbb{C})$ decompose into a finite
product of holomorphic maps sending $\mathbb{C}^n$ into unipotent subgroups in $\mbox{SL}_m(\mathbb{C})$?}

\medskip
  {Gromov's} interest in this question comes from the question about s-homotopies (s for spray). In this particular
example the spray  on $ \mbox{SL}_m(\mathbb{C})$ is that coming from the multiplication with
unipotent matrices. Of course one cannot use the upper and lower diagonal unipotent matrices only 
to get a spray (there is no submersivity at the zero section!), there need to be at least one more unipotent subgroup to be used in the multiplication. Therefore the factorization in a product of upper and lower diagonal matrices seems to be a stronger condition than to find a map into the iterated spray, but since all maximal unipotent subgroups in $\mbox{SL}_m(\mathbb{C})$ are conjugated and the upper and lower diagonal matrices generate $\mbox{SL}_m(\mathbb{C})$ these two problems are in fact equivalent. We refer the reader
for more information on the subject to   {Gromov's} above mentioned paper.

As an application of flexibility via a stratified version of the Oka principle from \cite{Forstratified} Ivarsson and the author  gave  a complete positive solution of \textsc{Gromov's}   {Vaserstein} problem.

\begin{Thm} \cite{IK}
Let $X$ be a finite dimensional reduced Stein space and $f\colon X\to \mbox{SL}_m(\mathbb{C})$ be a holomorphic mapping that is null-homotopic. Then there exist a natural number $K$ and holomorphic mappings $G_1,\dots, G_{K}\colon X\to \mathbb{C}^{m(m-1)/2}$ such that $f$ can be written as a product of upper and lower diagonal unipotent matrices
\begin{equation*}f(x) = \left(\begin{matrix} 1 & 0 \cr G_1(x) & 1 \cr \end{matrix} \right)   
\left(\begin{matrix} 1 & G_2(x) \cr 0 & 1 \cr \end{matrix} \right)  \ldots \left(\begin{matrix} 1 & G_K(x)\cr 0 & 1 \cr \end{matrix} \right) 
\end{equation*} for every $x\in X$.
\end{Thm}

We leave it as an exercise to the reader to factorize the Cohn example holomorphically.

\section{The holomorphic linearization problem} \label{Linearization}

The holomorphic linearization problem was well known to the German Complex Analysis School of Grauert and Remmert.
For example a  first special case of the Luna slice theorem was proved by Kuhlmann. The problem was officially stated
by Huckleberry in his chapter in the (that time still Soviet) Encyclopedia of Mathematical Sciences. \cite{Huck}

\begin{Prob}[\bf Holomorphic Linearization Problem] 

Suppose a reductive group $G$ is 
acting   holomorphically on $\C^n$, $n\ge 2$,
Does there exist a holomorphic change of variables $\alpha \in \Aut_{hol} (\C^n)$ such that $\alpha \circ G \circ \alpha^{-1} $ is linear?
\end{Prob}

Independently the late Walter Rudin got interested in this problem, he formulates it
more specifically for a finite cyclic group $G= \Z/n\Z$ in his paper with Ahern \cite{AR}.

\subsection{Basic notions and properties} \label{basics}

Two students of Alan T. Huckleberry, namely Dennis Snow and my adviser Peter Heinzner \cite{Heinzner}
developed the theory of reductive group actions on Stein spaces guided by the results
from algebraic geometry, most prominently by the slice theorem in \'etale topology due to  Domingo Luna. 


Let a reductive group $G=K^\C$, the universal complexification of its maximal compact subgroup $K$, act on a Stein space $X$. Identify two points  $x_1 \equiv x_2$ if $f(x_1) = f(x_2)$ for all $G$-invariant holomorphic functions $ f\in \cO^G (X)$. The quotient space $X\cat G$  has a natural complex structure\
such that $\pi : X \to X\cat G=: Q_X$ is holomorphic. The quotient space $Q_X$ (whose holomorphic functions are the invariant holomorphic functions on $X$, $\cO (Q_X) = \cO^G (X)$) is a  Stein space \cite{Snow}.  Each fiber of $\pi$ contains a unique closed $G$-orbit which is contained in the closure of any of the other $G$-orbits in the same $\pi$-fiber. For a point $x$ in this unique closed orbit the
complexification of the $K$-isotropy is equal to the $G$-isotropy, $(K_x)^\C = (K^\C)_x$.

\bigskip
If $X$ is affine algebraic and $G$ acts algebraically the categorical quotient is the same as the "Hilbert" quotient $(Spec (\C[X]^G) = X\cat G)$. 

\bigskip 
In general the fibers of $\pi$ are affine $G$-varieties not necessarily reduced. This can be easily seen from the fact that locally every Stein $G$-manifold can be properly equivariantly 
embedded into a linear $G$-representation, i.e., a $\C^N$ with linear $G$ action, together with fact 1 below. If $X$ has only  finitely many slice types there is even a global such embedding. This is the main result of Peter Heinzner's  Ph.D. thesis \cite{Heinzner1}. 

The following facts are easy. The first fact uses the property that $G$-invariant holomorphic functions from a subvariety $Y$ of a Stein space $X$ extend 
to $G$-invariant functions on $X$ (proved by using Cartan extension and averaging over the maximal compact subgroup $K$ of $G$). The second fact is obvious.

\smallskip\noindent
{\bf Fact 1:}  If $Y$ is a closed $G$-invariant subspace of a Stein $G$-space $X$, then the restriction of the categorical quotient map $\pi : X \to Q_X$ to $Y$ is a categorical quotient map for $Y$
             and the categorical quotient for $Y$ is the image $Q_Y \cong \pi (Y)$ of $Y$, which is a closed subspace of $Q_X$.

\smallskip\noindent
{\bf Fact 2:} If $X$ is a Stein $G$-space with categorical quotient map  $\pi : X \to Q_X$ and $Y$ is a Stein space with trivial $G$-action, then the categorical quotient map for $X\times Y$
is $\pi \times \Id_Y : X\times Y \to Q_X\times Y \cong Q_{X\times Y}$.

\begin{Exa} 

Let the group $\Sl_n (\C)$ act on  the vector space of all $n$ by $n$ matrices $Mat(n\times n, \C)$ by conjugation $\Sl_n (\C) \times Mat(n\times n, \C) \to Mat(n\times n, \C)$
$ (G,M) \mapsto GMG^{-1}$. We know that the orbits of this action correspond to Jordan normal forms. The categorical quotient is given by the coefficients of the characteristic polynomial
$\chi_M (\lambda) = \det (\lambda E-M) = \lambda^n + a_{n-1} \lambda^{n-1} + \ldots + a_0$ or equivalently by the elementary symmetric polynomials in the eigenvalues of the matrix $M$.
This means 

$$\pi: Mat(n\times n, \C) \to \C^n, \\\\ M \mapsto (a_0, a_1, \ldots, a_{n-1})$$

\noindent
is a categorical quotient  map for the above action and thus $Q_{Mat(n\times n, \C)} = Mat(n\times n, \C) \cat \Sl_n (\C) \cong \C^n$.
\end{Exa}

The categorical quotient carries a stratification defined as follows. We say that two points $q, q^\prime \in Q_X$ are in the same Luna stratum if the fibers $\pi^{-1} (q)$ and $\pi^{-1} (q^\prime)$ are $G$-biholomorphic. Remember the fibers are affine  not necessarily reduced $G$-varieties.  

\begin{Thm} [Luna stratification] \cite{Snow}
The Luna strata form a locally finite stratification of $Q_X$ by locally closed smooth analytic
subvarieties.
\end{Thm}

The following very easy example will play a role for our counterexamples to Holomorphic Linearization.

\begin{Exa}  \label{nonaction} Let $G=\C^\star$ act on $\C^3$ by the rule 

\begin{equation} G \times \C^3 \to \C^3, \ \ (\lambda, (u, v,  w)) \mapsto (\lambda^2 u, \lambda^{-2} v,  \lambda w)
\end{equation}

The categorical quotient $Q_{\C\star}$ is isomorphic to $\C^2$ and the map is $(u, v, w) \mapsto (uv, w^2 v)$. The Luna stratification of the quotient is easy:
$Q_{\C^\star}\cong \C^2 \supset \C\times\{0\}\supset \{(0,0)\}$. The isotropy groups in the closed orbits over points of the strata are isomorphic to $\{\Id\}, \{\pm \Id \}, \C^\star$ respectively.

\end{Exa}

Moreover the local structure of holomorphic $G$-actions on Stein manifolds is well understood, thanks to the following slice theorem due to Snow \cite{Snow}.  

The setting of this holomorphic slice theorem is as follows.  Let the complexification $G=K^\C$ of a compact Lie group $K$ act holomorphically on a Stein space $X$ with categorial quotient $\pi:X\to X\cat K^\C$.  Let $x$ be a point in the unique closed $G$-orbit in the fiber $\pi^{-1} (q)$, where $q= \pi (x)$, with stabilizer $L^\C$, where $L=K_x$.  Let $V=T_xX/T_x K^\C x$ be the normal space to the orbit $K^\C x$ at $x$.  It is an $L^\C$-module.  With respect to the identification $K^\C/L^\C\simeq K^\C x$, the normal bundle $N$ of $K^\C x$ in $X$ is isomorphic to $K^\C\times^{L^\C} V$.

\begin{Thm}[Luna slice theorem]   \label{thm:holomorphic.slice}
There is a  saturated (with respect to $\pi$) Stein neighborhood $U$ of the orbit $K^\C x$ in $X$, $K^\C$-equivariantly biholomorphic to a subvariety $A$ of a neighborhood of the zero section of $N$.  The embedding $\iota:U\to N$ maps $K^\C x$ biholomorphically onto the zero section of $N$.  
\end{Thm}

If $X$ is smooth, the submanifold $A$ in the Luna Slice  Theorem is open. Thus the representation of $L^\C$ on $V=T_xX/T_x K^\C x$, called the slice representation, determines the local structure of the $K^\C$-manifold $X$
on a saturated neighborhood of $x$. Therefore our  Luna stratification is the same as the stratification by slice representations from \cite{Snow}. We will call the categorical quotient $Q_X$ equipped with the Luna stratification  the {\bf Luna quotient}. An isomorphism of Luna quotients $Q_X$ and $Q_Y$ (for Stein $G$-manifolds $X$ and $Y$) is by definition a biholomorphism between the Stein spaces $Q_X$ and $Q_Y$
mapping corresponding Luna strata onto each other.  Clearly an equivariant biholomorphism between $X$ and $Y$ induces an isomorphism of the Luna quotients.

\subsection{History of the linearization problem} \label{history}

Let us describe the most important results of the holomorphic part of the problem

\begin{itemize}
\item
If the quotient $\C^n \cat G$ is zero dimensional (one point) the action is linearizable. This follows from Luna's slice theorem (Theorem \ref{thm:holomorphic.slice}).

\item
Holomorphic $\C^\star$-actions on $\C^2$ are linearizable (M. Suzuki 1977 \cite{Suzuki}). 

\item
Holomorphic actions with one-dimensional quotient are linearizable. This is the main result from the Ph.D. thesis of  Jiang 1990.  A new proof  is given by  L\'arusson, Schwarz  and the author in 2016.
\cite{KLS2}.

\item
The first counterexamples to the {\bf Algebraic Linearization Problem}
are given in 1989 by G. Schwarz: These examples are  $G$-vector bundles over representation spaces, see Example \ref{counter} and the discussion preceding it.
Based on the same method F. Knop constructs   {\bf algebraically} non-linearizable actions for all semi-simple groups $G$  \cite{Knop}.

\item
Holomorphic $G$-vector bundles over representation spaces are holomorphically trivially, in other words: the corresponding actions are  linearizable. 
This is an application of the Equivariant Oka Principle proved by Heinzner and the author in 1995. For the definition of a Kempf-Ness set we refer the interested
reader to  \cite{HK}, for our application the existence of the homotopy is not essential.  The existence of the homotopy over the whole space $X$ together
with a generalization to bundles with homogeneous fiber see the recent work of L\' arusson, Schwarz and the author \cite{KLS.homog}.

\begin{Thm}[Equivariant Oka Principle]   \label{thm:HK-classification}

{\rm (a)}  Every topological principal $K$-$G$-bundle on A Stein space $X$ is topologically $K$-isomorphic to a holomorphic principal $K^\C$-$G$-bundle on $X$.

{\rm (b)}  Let $P_1$ and $P_2$ be holomorphic principal $K^\C$-$G$-bundles on $X$. Let $c$ be a continuous $K$-equivariant section of $\Iso(P_1,P_2)$ over $R$. Then there exists a homotopy of continuous $K$-equivariant sections $\gamma(t)$, $t\in[0,1]$, of $\Iso(P_1,P_2)$ over  a Kempf-Ness set $R$ such that $\gamma(0)=c$ and $\gamma(1)$ extends to a holomorphic $K$-equivariant isomorphism from $P_1$ to $P_2$.
\end{Thm}

\item
Holomorphic actions of de Jonqui\`eres (triangular) type (see formula \eqref{Jonq}) on $\C^n$ are linearizable. Also holomorphic actions in the overshear group of $\C^2$ are linearizable. This result of Kraft and the author \cite{KK} also makes use of  the 
Equivariant Oka Principle (Theorem \ref{thm:HK-classification}).

\end{itemize}

\subsection{Counterexamples}\label{counterex}

The first counterexamples to the Holomorphic Linearization Problem were found by Derksen and the author. In contrast to the algebraic situation, where
 counterexamples for abelian groups are still unknown, they constructed counterexamples for any reductive group $G$. For the case of finite groups  $G$ they
had to use the classification of finite simple groups.

\begin{Thm} \cite{DK1}
For all reductive groups $G$ there exists a number $N(G)$ such that for all $n\ge N(G)$ there is a non-linearizable 
action of $G$ on $\C^n$.
\end{Thm}

The smallest dimension $N(G)$ in which we know counterexamples is 4. These are counterexamples   for $G=\C^\star$ or $G=Z / 2 Z$. The exact value of $N(G)$ is not known
for a single group $G$. Funny enough the counterexample for $G=\Z / 2 \Z$ is not explicit. There are two examples: One is an action on $\C^4$, the other one on a certain
$4$-dimensional manifold $Y$. If the first example is linearizable, then the manifold $Y$ is biholomorphic to $\C^4$ and the action on it is not linearizable. We will come back to this 
topic, when discussing the relation to famous open problems.

Of course we want to give the reader an impression how these counterexamples are constructed. Hereby we will limit  our presentation to the case of $G=\C^\star$. The main idea, originating from the work of Asanuma \cite{As}, is to use non-straightenable holomorphic embeddings of $\C^l$ into $\C^n$ to
produce an action on some $\C^N$ which cannot be linearizable by the following reason:

\smallskip\noindent
 {\sl The Luna quotient of this action on $\C^n$ is not isomorphic to the Luna quotient of any linear $G$-action on $\C^N$}.

\smallskip
Let us describe the method from \cite{DK1} and \cite{DK} (based on \cite{As}) to construct (non-linearizable) $\C^*$-actions on
affine spaces out of (non-straightenable) embeddings $\C^l\hookrightarrow \C^n$. At the same time we want to present a strengthening of the original method to a parametrized version. It will turn out  that if the embeddings are holomorphically parametrized, then the resulting $\C^*$-actions depend holomorphically on the parameter. Theorem \ref{mainembedding} will give us then the following quite striking  result due to Lodin and the author \cite{Kutzschebauch-Lodin}, a whole family of counterexamples
to linearization:

\begin{Thm} \label{Lodin}

For any $n\ge 5$ there is a holomorphic family of $\C^*$-actions on $\C^n$ parametrized by $\C^{n-4}$
$$\C^{n-4} \times \C^* \times \C^n \to \C^n \quad(w, \theta, z) \mapsto \theta_w (z)$$ so that for
different parameters $w_1\neq w_2\in \C^{n-4}$ there is no equivariant isomorphism between the actions $\theta_{w_1}$ and  $\theta_{w_2}$.

Moreover for $n\ge 5$ there are such $\C^*$-actions on $\C^n$  parametrized by $\C$ with the additional property that $\theta_0$ is a linear action.
\end{Thm}

Let's go through the method:
For an embedding $\varphi: \C^l \to \C^n$ take generators of the ideal $I_{\varphi (\C^l)} < \cO (\C^n)$ of the image manifold, say 
$f_1, \ldots, f_N \in \cO ({\C^n)}$ (in this case $N= n-l$ would be sufficient, since $\C^l$ is always a
complete intersection in $\C^n$ by results of Forster and Ramspott \cite{FoRa}, but this is not important for the construction) and consider the manifold

\begin{multline*}
M:= \{ (z_1, \ldots, z_n, u_1,\ldots u_N, v) \in \C^{n+N+1} : 
f_i (z_1, \ldots, z_n) = u_i \ v      \quad \forall \ i=1, \ldots, N \}
\end{multline*}

\noindent
which in \cite{DK1} is called Rees space. This notion was introduced there by the authors since they were not aware of the fact that this is a well-known construction, called affine modification, going back to Oscar Zariski. Geometrically the manifold $M$ results from $\C^{n+1}_{z, v}$ by blowing up along the center $\mathcal{C} = \varphi (\C^l) \times 0_v$ and deleting the proper transform of the divisor
$\mathcal{D} = \{ v = 0\}$. Since our center is not algebraic but analytic,  the process usually is called
pseudo-affine modification.

Let's denote the constructed manifold $M$ by $Mod (\C^{n+1}, \mathcal{D}, \mathcal{C})  = 
Mod( \C^{n+1}_{z, v}, \{v = 0\}, \varphi (\C^l)\times \{v=0\})$. It's clear from the geometric description that
the resulting manifold does not depend on the choice of generators for the ideal $I_\mathcal{C}$ of the center. 
The important fact about the above modifications is that 

\smallskip\noindent
$$Mod( \C^{n+1}_{z, v}, \{v = 0\}, \varphi (\C^l)\times \{v=0\}) \times \C^l$$ is biholomorphic to $$\C^{n+l+1}
\cong Mod (\C^{n+l+1}_{z, u, v}, \{ v=0\}, \varphi(\C^l) \times 0_u \times 0_v).$$ The later biholomorphism
comes from the fact that there is an automorphism  of $\C^{n+l+1}$ leaving the divisor $\{ v= 0\}$ invariant
and straightening the center $\varphi(\C^l) \times 0_v$ inside the divisor (see Lemma 2.5. in \cite{DK1}). This is the so called
Asanuma trick. Let's present  this important fact with holomorphic dependence  on a parameter.

\begin{Lem}\label{straight}
Let $\Phi_1 : \C^k \times X \hookrightarrow \C^k \times \C^n$, $\Phi_1 (w, x) = (w, \varphi_1 (w, x))$ and
$\Phi_2 : \C^k \times X \hookrightarrow \C^k \times \C^m$, $\Phi_2 (w, x) = (w, \varphi_2 (w, x))$ be two holomorphic families of proper holomorphic embeddings of a complex space $X$ into $\C^n$ resp. $\C^m$ parametrized by $\C^k$. Then there is an automorphism $\alpha$ of $\C^{n+m}$ parametrized by $\C^k$, i.e., $\alpha \in \Aut_{\rm hol} (\C^k_w \times \C^{n+m}_z)$ with $\alpha (w, z) = (w, \tilde \alpha (w, z))$, such that $\alpha \circ (\Phi_1 \times 0_m) = 0_n \times \Phi_2$.
\end{Lem}

\begin{proof}
By an application of Theorem B the holomorphic map $\varphi_1 : \C^k \times X $ to $ \C^n$ extends to a holomorphic map $\mu_1$ from $\C^k \times \C^m \supset \Phi_2 (\C^k \times X)$ to $\C^n$ (so $\mu_1\circ \varphi_2 = \varphi_1$). Likewise there is a holomorphic map $\mu_2 : \C^k\times \C^n \to \C^m$
with $\mu_2 \circ \varphi_1 = \varphi_2$. Define the parametrized automorphisms $\alpha_1, \alpha_2$
of $\C^k\times \C^n\times \C^m$ by $\alpha_1 (w, z, y) = (w, z, y+ \mu_2 (w, z))$  and 
$\alpha_2 (w, z, y) = (w, z + \mu_1 (w, y), y)$. Now $\alpha = \alpha_2^{-1} \circ \alpha_1$ is the desired 
automorphism.
\end{proof}

\begin{Lem} 
\label{fam}
Let $\Phi : \C^k \times \C^l \hookrightarrow \C^k \times \C^n$ $\Phi (w, \theta) = (w, \varphi (w, \theta))$ be a holomorphic family of proper holomorphic embeddings of $\C^l$ into $\C^n$ parametrized by $\C^k$.

Then $Mod( \C^{k+n+1}_{w, z, v}, \{ v=0\}, \Phi (\C^k\times \C^l )\times \{ v=0 \}) \times \C^l \cong \C^{k+n+l+1}$. Moreover there is a  biholomorphism such that the restriction to each fixed parameter  $w\in \C^k$ is a biholomorphism from $ Mod( \C^{n+1}_{ z, v}, \{v=0\}, \Phi(\{w\}\times \C^l)\times \{v=0\}) \times \C^l \cong \C^{n+l+1}$.
\end{Lem} 

\begin{proof}
Apply Lemma \ref{straight} to the families $\Phi_1 = \Phi$ and $\Phi_2$ the trivial family
$\Phi_2 : \C^k \times \C^l \hookrightarrow \C^k\times \C^l$ $\Phi_2 (w, \theta) = (w, \theta)$. Let $\alpha
\in \Aut_{\rm hol} (\C^k \times \C^n \times \C^l)$ be the resulting parametrized automorphism which we extend to $\C^{k+n+l+1}$ by letting it act trivial on the last coordinate $v$. Then by definition
$Mod( \C^{k+n+1}_{w, z, v}, \{ v=0\}, \Phi (\C^k\times \C^l )\times \{ v=0 \}) \times \C^l = 
Mod( \C^{k+n+l+1}_{w, z,\theta,  v}, \{ v=0\}, \Phi (\C^k\times \C^l )\times \{ v=0 \} \times 0_l)$ and applying (the extended) $\alpha$ we get that the later is biholomorphic to
$Mod( \C^{k+n+l+1}_{w, z,\theta,  v}, \{ v=0\}, \C^k_w \times 0_n \times \C^l_\theta \times \{ v=0 \} )$. The last manifold is obviously biholomorphic to $\C^{k+n+l+1}$ since blowing up along a straight
center and deleting the proper transform of a straight divisor does not change the affine space.
The above constructed biholomorphism restricts to each fixed parameter as desired since $\alpha$ is
a parametrized automorphism. This can be also seen by writing down concrete formulas for the modifications using generators $f_1 (w,z), \ldots, f_N(w, z)$ of the ideal $I_{\phi(\C^k\times\C^l)}$ in
$\cO (\C^{k+n})$  and remarking that for each fixed $w \in \C^k$ the functions $f_1 (w,\cdot), \ldots, f_N(w, \cdot )$ generate the ideal $I_{\Phi_w (\C^l)}$.
\end{proof}

Now we describe the group actions:

\noindent
Let  $f_1 (w,z), \ldots, f_N(w, z)$ be generators of the ideal $I_{\phi(\C^k\times\C^l)}$ in
$\cO (\C^{k+n})$ and consider $Mod( \C^{k+n+1}_{w, z, v}, \{ v=0\}, \Phi (\C^k\times \C^l )\times \{ v=0 \}) \times \C^l \cong \C^{k+n+l+1}$ as the affine manifold given by equations:

\begin{multline*} 
\{ (w, z, v, u) \in \C^k\times \C^n\times \C\times \C^N :
f_i (w, z) = u_i \ v      \quad \forall \ i=1, \ldots, N \} \times \C^l_x
\end{multline*} 

On it we consider the action of $\C^*_\nu$ given by the restriction of the following linear action on the ambient space:

\begin{multline}
\C^* \times \C^k \times \C^n \times \C \times \C^N \times \C^l \to \C^k \times \C^n \times \C \times \C^N \times \C^l \\(\nu, (w, z, v, u, x)) \mapsto (w, z, \nu^2 v, \nu^{-2} u_1, \ldots ,\nu^{-2} u_N, \nu x_1, \ldots, \nu x_l)
\end{multline}

This gives by Lemma \ref{fam} a holomorphic family of $\C^*$-actions on $\C^{n+l+1}$ parametrized by $\C^k$, i.e., an action
$\C^* \times \C^k \times \C^{n+l+1} \to \C^k \times \C^{n+l+1}$ of the form $(\nu (w, z)) \mapsto (w, \nu (w, z))$. Calculating (as in \cite{DK}) the Luna-stratification of the categorical quotient $\C^{n+l+1}/  \hspace{-3 pt} / \C^*$ for the $\C^*$-action  for fixed $w$, in particular the inclusion of the fixed point stratum in the $\Z/2\Z$-isotropy stratum one sees that this inclusion is biholomorphic to $\Phi_w (\C^l) \subset \C^n$. 
The reader is invited to do this calculation using Example \ref{nonaction} and Fact 1 from section \ref{basics}.
Thus if
for different parameters $w_1 \ne w_2$ there were an equivariant automorphism $\alpha \in \Aut_{\rm hol} (\C^n)$, the induced isomorphism of the categorical quotients would map the Luna-stratifications onto each other. Therefore the restriction of that induced isomorphism to the $\Z/2\Z$-isotropy stratum would
give an automorphism $\beta$ of $\C^n$ with $\beta (\Phi_{w_1} (\C^l)) = \Phi_{w_2} (\C^l)$. This shows
that pairwise non-equivalent embeddings lead to non-equivalent $\C^*$-actions. This concludes the proof of Theorem \ref{Lodin} except for the moreover part. 
The proof of this fact is a simple trick of contracting the parameter space. We refer the reader to \cite{Kutzschebauch-Lodin}. Also it is clear from the above discussion that one can construct uncountably many non-linearizable $\C^\star$-actions on  $\C^4$ using the last assertion from Theorem \ref{mainembedding}.

\subsection{Relation to famous problems} \label{relations}
The Holomorphic Linearization Problem is connected to famous problems about complex affine space $\C^n$. The first one is the holomorphic version
of the Zariski Cancellation Problem, a problem which is still open in both the algebraic category over $\C$ and in the holomorphic category.

\begin{Prob}[Zariski Cancellation] \label{Zariski}
If $X$ is a complex manifold such that $X\times \C$ is biholomorphic to $\C^{n+1}$.
Does it follow that $X$ is biholomorphic to $\C^n$?
\end{Prob}

There is an easy connection to Linearization since if we had a counterexample $X$ to the Zariski Cancellation Problem, say $\dim X =n $ then on $\C^{n+1} \cong X\times \C$
the $\Z/2\Z$-action given by $\Z/2\Z \times (X\times \C) \to X\times\C$, $(\sigma, (x, t)) \mapsto (x, -t)$ would have a fixed point set $X$ which would not be biholomorphic to
an affine space. Clearly fixed point sets of linear actions are affine spaces, thus the $\Z/2\Z$-action would be non-linearizable. Another less obvious connection comes from the
construction of our counterexamples, the Asanuma trick. Set $X=  Mod( \C^{n+1}_{ z, v}, \{v=0\}, \Phi( C^l)\subset \{v=0\})$ for a non-straightenable holomorphic embedding $\Phi : \C^l \to
 \C^n$. By Lemma \ref{fam} 
$X \times \C^l \cong \C^{n+l+1}$. If $X$ were biholomorphic to $\C^{n+1}$ we would have non-linearizable actions in lower dimensions. But if $X$ were not biholomorphic to $\C^{n+1}$
there would be a counterexample to Zariski Cancellation.

Here is a connection to another well known problem, formulated as a conjecture by Varolin and Toth.
\begin{Prob}[Varolin-Toth-Conjecture] \label{VTconjecture} If a Stein manifold $X$ is diffeomorphic to $\R^{2n}$ ($n\ge 2$) and has the density property, 
is  $X$ then biholomorphic to $\C^n$?
\end{Prob}

Again our $X=  Mod( \C^{n+1}_{ z, v}, \{v=0\}, \Phi(\{\C^l)\subset \{v=0\})$ is the candidate. From $(3)$ in our list of examples in section \ref{examples} we see that it has the density property.
Also it can be proven that $X$ is diffeomorphic to $\R^{2n+2}$, see \cite{KaKu2}. Were it not biholomorphic to affine space, we would have a counterexample to the Varolin-Toth-Conjecture. There are more candidates for counterexamples to this conjecture. A famous one is the Koras-Russell threefold from equation \eqref{KR}, which is well known to be diffeomorphic to $\R^6$
and has the density property, see $(5)$ in our list of examples in section \ref{examples}.

\subsection{Optimal positive results}
The way of constructing counterexamples to the Holomorphic Linearization Problem was providing group actions on $\C^n$ whose Luna quotient is not isomorphic
to the Luna quotient of a linear action. This raises the following natural question:

\medskip\noindent
\textbf{Question 1:} 
If  the Luna quotient  of an action of a reductive group $G$ on $\C^n$ is biholomorphic to the quotient of a linear action, does it follow that the action is linearizable?

\medskip
In general one can replace $\C^n$ by arbitrary Stein manifolds. 

\medskip\noindent
\textbf{Question 2:} If two Stein $G$-manifolds have isomorphic Luna quotients, under which additional assumptions are they $G$-biholomorphic?

\medskip
The following example shows that there is at least a topological obstruction for this to hold true.

\begin{Exa} Take any Stein manifold  $M$ which admits 2 non-isomorphic holomorphic line bundles $\pi_1 : L_1 \to M$ and $\pi_2 : L_2\to M$. Remember that   by Oka principle  $H^2(M, \Z)$ parametrizes the 
holomorphic line bundles on $M$, so this cohomolgy group has to be non-trivial. Consider the action of $\C^\star$ on the total spaces $X=L_1$ and  $Y=L_2$ of the line bundles by fibre wise multiplication. 
The categorical quotient maps for these actions are just the bundle projections.  Both categorical quotients are just isomorphic to $M$.The fibers  of the categorical quotient maps consist of 2 orbits, a fixed point (contained in the zero section of the bundle) and the rest of the line, a free $\C^\star$ orbit. The Luna stratification is trivial, with one stratum $M$. 
A $\C^\star$-equivariant biholomorphism between $L_1$ and $L_2$ is linear in the fibers of the line bundle, thus a bundle isomorphism. But the line bundles are non-isomorphic. The obstruction 
to an equivariant biholomorphism under the existence of a Luna biholomorphism in this
example is purely topological: $H^2(M, \Z)$.
\end{Exa}

We are now going to present recent results on the above two questions obtained  by L\'arusson, Schwarz and the author, all results are from the papers \cite{KLS1}, \cite{KLS2}, \cite{KLS3}.
The setting is as follows: Let  $X$ and $Y$ be  Stein manifolds  on which $G$ acts holomorphically.  
 We have quotient mappings $p_X\colon X\to Q_X$ and $p_Y\colon Y\to Q_Y$ where $Q_X$ and $Q_Y$ are normal Stein spaces, the categorical quotients of $X$ and $Y$
 Suppose there is a biholomorphism $\phi\colon Q_X\to Q_Y$ which preserves the Luna strata, i.e., $X_q$ is $G$-biholomorphic to $Y_{\phi(q)}$ for all $q\in Q_X$.  We say that \emph{$X$ and $Y$ have common quotient $Q$}.

Set
$$
\Iso(X,Y)=\prod_{q\in Q}\Iso(X_q,Y_q)
$$ 
where $\Iso(X_q,Y_q)$ denotes the set of $G$-biholomorphisms of $X_q$ and $Y_q$. In general, there is no reasonable structure of complex variety on $\Iso(X,Y)$ as examples in \cite{KLS2} show.

Let $\Phi\colon X\to Y$  be a $G$-diffeomorphism inducing the identity on the quotient.
$\Phi$ is  called {\bf strict} if it induces a $G$-biholomorphism of
 $(X_q)_{red}$ with $(Y_q)_{red}$  for all $q\in Q$.  One can think of the strict $G$-diffeomorphisms as  ``smooth sections of $\Iso(X,Y)$" (although the later has no good structure).

\begin{Thm} \label{KLSmain}
Let $X$ and $Y$ be Stein $G$-manifolds with common quotient $Q$. Suppose that there is a strict $G$-diffeomorphism $\Phi\colon X\to Y$. Then $\Phi$ is homotopic, through strict $G$-diffeomorphisms, to a $G$-biholomorphism from $X$ to $Y$.
\end{Thm}

We would like to comment on this theorem whose proof has two major steps. Our aim is to lift a Luna isomorphism $\phi$ between  $Q_X$ and $Q_Y$ in the diagram

$$\begin{matrix}   X & \xlongrightarrow{\makebox[2cm] {}} & Y  \\
         \ \ \ \   \downarrow{p_X} &                       &   \ \ \ \  \downarrow{p_Y}         \\
\ \  Q_X & \xlongrightarrow{\makebox[2cm]{ $\phi$ }} &  \ \  Q_Y \\ \end{matrix} $$
 
\noindent
to a $G$-equivariant biholomorphism between $X$ and $Y$ (under the assumption that there is a lift to a strict $G$-diffeomorphism). The first step which may readily be overlooked is the local lifting. For each $q \in Q_X$ and $\phi (q) \in Q_Y$
there are open neighborhoods  $U \subset Q_X$ and $V = \varphi (U)  \subset Q_Y$ whose preimages under the categorical quotient maps are described by the Luna slice theorem  Theorem \ref{thm:holomorphic.slice}. And since $\phi$ respects the Luna stratification, both $p_X^{-1} (U)$ and $p_Y^{-1} (V)$ are described by the same slice model and thus are $G$-biholomorphic.
The problem is that this existing $G$-biholomorphism need not to be a lift of $\phi$, it may induce another local Luna isomorphism between $U$ and $V$. Therefore our first step is non-trivial and
uses the additional information, namely the existence of a strict $G$-diffeomorphism. 

The second step is then  the gluing of the local lifts of the Luna isomorphism to a global lift. This is an Oka principle, which is a separate important result:

Suppose that we have a stratified biholomorphism $\phi\colon Q_X\to Q_V$ where $V$ is a $G$-module. Again we  identify $Q_X$ and $Q_V$ and call the common quotient $Q$. We have quotient mappings $p\colon X\to Q$ and $r\colon V\to Q$. Assume there is an open cover $\{U_i\}_{i\in I}$ of $Q$ and $G$-equivariant biholomorphisms $\Phi_i:p^{-1}(U_i)\to r^{-1}(U_i)$ over $U_i$ (meaning that $\Phi_i$ descends to the identity map of $U_i$).   We express the assumption by saying that $X$ and $V$ are \textit{locally $G$-biholomorphic over a common quotient}.   
Equivalently, our original $\phi\colon Q_X\to Q_V$ locally lifts to $G$-biholomorphisms of $X$ to $V$.

\begin{Thm} 
Let $X$ and $Y$ be Stein $G$-manifolds which locally $G$-biholomorphic over a  common quotient $Q$. Suppose that there is a strict $G$-diffeomorphism $\Phi\colon X\to Y$. Then $\Phi$ is homotopic, through strict $G$-diffeomorphisms, to a $G$-biholomorphism from $X$ to $Y$.
\end{Thm}
 
In the case of a generic action (see Definition \ref{large} below) this theorem can be deduced from the Equivariant Oka Principle (Theorem \ref{thm:HK-classification}).  The proof in the general case is much more involved.

Together with the above  described first step this proves Theorem \ref{KLSmain}. The additional obstruction, besides  the Luna-isomorphism of the quotients, is the strict $G$-diffeomorphism.
We also have a result, where the obstruction is more topological (instead of smooth). This is the existence of a strong $G$-homeomorphism, which is the right version of 
``continuous sections of $\Iso(X,Y)$". Since the definition of strong $G$-homeomorphism is not so straightforward and the result includes an extra assumption on the common quotient we refer the interested reader to \cite{KLS2} for details.

Now let us get back to the special case of linearization, i.e., one of the $G$-manifolds, say $Y$, is a linear representation of $G$, i.e., a $G$-module. This of course gives immediately more information on the Luna quotient
and there is hope that the additional obstruction just disappears because of the simple topology or diffeomorphism type of $\C^n$. We have not been able to confirm this hope completely, but substantial results are proven. Remember that the proof of Theorem \ref{KLSmain} consisted of two steps, the first step, the local $G$-diffeomorphisms are the problem, the second step is solved:
 
\begin{Thm}
Suppose that $X$ is a Stein $G$-manifold, $V$ is a $G$-module and $X$ and $V$ are locally $G$-biholomorphic over a common quotient. Then $X$ and $V$ are $G$-biholomorphic.
\end{Thm}

For the local $G$-isomorphisms we still do not know the optimal result. We have to make an additional technical assumption on the representation.

\begin{Def} \label{large}
 Assume that the set of closed orbits with trivial isotropy group is open in $X$  and that the complement, a closed subvariety of $X$, has complex codimension at least two. We say that $X$ is \emph{generic\/}. Let $X_{(n)}$ denote the subset of $X$ whose isotropy groups have dimension $n$. We say that $X$ is  \emph{large\/} if $X$ is generic and $\codim X_{(n)}\geq n+2$ for $n\geq 1$. 
 \end{Def}

For a simple group all but finitely many irreducible representations are large. A representations is large if all irreducible factors  are large. 

\begin{Thm}
Suppose that $X$ is a Stein $G$-manifold and $V$ is a $G$-module satisfying the following conditions.
\begin{enumerate}
\item There is a stratified biholomorphism $\phi$ from $Q_X$ to $Q_V$.
\item  $V$  (equivalently, $X$) is large.
\end{enumerate}
Then, by perhaps changing $\phi$, one can arrange that $X$ and $V$ are locally $G$-biholomorphic over $Q_X\simeq Q_V$, hence $X$ and $V$ are $G$-biholomorphic.
\end{Thm}

Some special cases where the representation is not large can be dealt with as well:

\begin{Thm}
Suppose that $X$ is a Stein $SL_2 (\C)$-manifold and $V$ is a $SL_2 (\C)$-module. If there is a  stratified biholomorphism  (Luna isomorphism) $\phi$ from $Q_X$ to $Q_V$,
then $X$ is $SL_2 (\C)$-biholomorphic to $V$.
\end{Thm}

Also an old result of Jiang from our history of the Holomorphic Linearization Problem can be reproved by using this approach.

\begin{Thm} Suppose that $X$ is a Stein $G$-manifold and $V$ is a $G$-module with one-dimensional categorical quotient $Q_V$. If there  is a stratified biholomorphism $\phi$ from $Q_X$ to $Q_V$,
then $X$ is $G$-biholomorphic to $V$.
\end{Thm}

A comment on the last three theorems is in order, namely that we do not assume from the beginning that the Stein manifold $X$ is biholomorphic to affine space. Therefore these theorems can be viewed as a characterization of
$\C^n$. On the other hand it seems unlikely that this characterization can be applied  in any interesting case to prove that some Stein manifold is $\C^n$. 
For example let's pretend we would like to use it to prove  that the Koras-Russel cubic threefold $M_{KR}$ given by equation \eqref{KR} is biholomorphic to $\C^3$. The only known action of a reductive group on $M_{KR}$ is the famous $\C^\star$-action given by the restriction to $M_{KR}$ of the linear action 
$\C^\star \times \C^4 \to \C^4$ 
$$ (\lambda, (x, y, s, t)) \mapsto (\lambda^6 x, \lambda^{-6} y, \lambda^3 s, \lambda^2 t).$$

 The reader is invited to calculate
the Luna quotient and to observe that this Luna quotient is not isomorphic to the Luna quotient of a linear action. The theorem cannot be applied, and thus no
information whether $M_{KR}$ is biholomorphic to $\C^3$ or not can be obtained. On the other hand, if there were a(nother) way to conclude that $M_{KR}$
is biholomorphic to $\C^3$, we would have found a non-linearizable holomorphic $\C^\star$-action on $\C^3$. Up to now the holomorphic linearization for
$\C^\star$ is known to hold on $\C^2$, not to hold on $\C^n$ for $n\ge 4$ (see section \ref{history}) and is open on $\C^3$.

The last three theorems give already a lot of evidence that the answer to Question 1 could be positive. If so, this would in some sense give one  beautiful answer to the Holomorphic Linearization Problem. On the other hand the question about  minimal dimension for non-linearizable actions of a given reductive group $G$ is much more difficult to answer. Until we have an effective criterion which can distinguish affine space $\C^n$ among Stein manifolds with density property there is no hope for an answer.  At this point it is worth mentioning another  criterion for characterization of $\C^n$ obtained by Isaev and Kruzhilin. A Stein manifold $X$ of dimension $n$ whose holomorphic automorphism group $\Aut_{hol} (X)$ is  isomorphic as topological group to $\Aut_{hol} (\C^n)$ is biholomorphic to $\C^n$ \cite{IKruz}. This  conclusion even holds without the Stein assumption \cite{IKruz1}. Unfortunately this criterion is equally not applicable to the above problems as ours from the last three theorems. In the early 1990's  J.P. Rosay asked the author whether $\Aut_{hol} (\C^n)$ as an abstract group (without topology) could characterize $\C^n$. Unfortunately even a positive answer to this seemingly difficult question would not be very helpful for our problems.



\newcommand{\etalchar}[1]{$^{#1}$}
\providecommand{\bysame}{\leavevmode\hbox to3em{\hrulefill}\thinspace}
\providecommand{\MR}{\relax\ifhmode\unskip\space\fi MR }
\providecommand{\MRhref}[2]{%
  \href{http://www.ams.org/mathscinet-getitem?mr=#1}{#2}
}
\providecommand{\href}[2]{#2}

\end{document}